\documentclass[14pt,utf8,a4paper,french,upmath, natbib]{aclart}

\usepackage{aclraccourcis,hyperref}

\def\Res{\operatorname{Res}\nolimits}

\itemsep\smallskipamount
\usepackage{tikz}

\title{Balade newtonienne entre analyse et arithmétique}
\subtitle{Journées X-UPS 2023}
\author{Antoine Chambert-Loir}
\address{%
Université Paris Cité, Institut de Mathématiques de Jussieu-Paris Rive Gauche, F-75013, Paris, France}
\email{antoine.chambert-loir@u-paris.fr}

\begin{abstract}
Inventés par Kurt Hensel à la toute fin du 19e siècle sur le modèle
des séries en une indéterminée, les nombres $p$-adiques sont devenus
non seulement un outil indispensable de l'arithmétique contemporaine,
mais un sujet d'étude en soi. Dans ce texte, issu de deux exposés
aux journées X-UPS 2023, j'expliquerai
leur construction, comment ils s'insèrent dans un cadre plus vaste,
entre analyse et arithmétique, où deux constructions portant le nom
d'Isaac Newton jouent un rôle central: la méthode de Newton et la
notion de polygone de Newton.
\end{abstract}

\begin{altabstract}
Invented by Kurt Hensel at the very end of 19th century on the model of power series in one indeterminate, the $p$-adic numbers have not only become an indispensable tool of contemporary arithmetic, but a research topic per se. In this text, stemming out two talks at the 2023 X-UPS lectures, I shall explain their construction, how they fit in a vaster framework, between analysis and arithmetic, where two constructions bearing the name of Isaac Newton play a central role : Newton's method and Newton's polygon.
\end{altabstract}
\begin{document}
\maketitle

\tableofcontents

\section{Valeurs absolues, topologies, exemples}

\begin{defi}[\citealp{Kurschak-1913}]\label{defi.va}
Soit $A$ un anneau commutatif. 
On appelle \emph{valeur absolue}, ou \emph{semi-norme multiplicative},
sur~$A$ une application $m\colon A\to\R_+$ vérifiant les propriétés suivantes:
\begin{enumerate}
\item \label{va-01}
On a $m(0)=0$ et $m(1) = 1$ \emph{(normalisation)};
\item \label{va-mul}
Pour $a,b \in A$, on a $m(ab) = m(a) m(b)$
\emph{(multiplicativité)};
\item \label{va-it}
Pour $a,b\in A$, on a $m(a+b)\leq m(a) + m(b)$
\emph{(inégalité triangulaire)}.
\end{enumerate}
\end{defi}

\begin{exem}
Les lecteurices de ce texte connaissent probablement
deux exemples de valeurs absolues:
la valeur absolue usuelle sur le corps~$\R$
des nombres réels, et le module sur le corps~$\C$
des nombres complexes.
\end{exem}

\subsection{}
Soit $A$ un anneau commutatif et soit $m$ une valeur absolue sur~$A$.

Supposons que $A$ est un corps.
Alors, $0$ est le seul élément de~$A$ de valeur absolue nulle.
En effet, si $a \in A$ est non nul, on a $1=m(1)=m(a\cdot a^{-1})
=m(a)\cdot m(a^{-1})$, donc $m(a)\neq0$.

C'est le cas le plus important, et l'on peut s'y ramener comme suit.
Appelons \emph{noyau} de la valeur absolue
l'ensemble~$P$ des $a\in A$ tels que $m(a)=0$;
c'est un idéal premier de~$A$. Il existe sur l'anneau quotient~$A/P$
une unique application  telle que $m(\overline a)=m(a)$
pour tout $a\in A$. Cette application est encore une valeur absolue.
Par construction, elle ne s'annule qu'en~$0$.
Il existe alors sur le corps des fractions~$F$ de~$A/P$
une unique application~$m$ telle que 
$m(\overline a/\overline b)=m(\overline a)/m(\overline b)$
pour tous $\overline a,\overline b\in A/P$ tels que $m(\overline b)\neq0$.
Cette application est encore une valeur absolue.

\subsection{}
La motivation initiale de~\cite{Kurschak-1913}
pour l'introduction de la notion
de valeur absolue sur un corps commutatif~$F$ est qu'elle
permet de définir une notion de limite sur~$F$. 
Plus d'un siècle plus tard, expliquons comment faire 
dans le cas plus général d'un anneau.

Soit donc~$m$ une valeur absolue sur un anneau~$A$.
De la définition, on déduit que $(x,y)\mapsto m(x-y)$
est une semi-distance sur~$A$:
la symétrie découle de ce que $1=m(1)=m(-1)^2$
(propriété~(\ref{defi.va}/\ref{va-mul})), donc $m(-1)=1$,
et l'inégalité triangulaire est la propriété~(\ref{defi.va}/\ref{va-mul}).
Cette semi-distance est une distance si et seulement si $m$
ne s'annule qu'en~$0$.

Pour $a\in A$ et $r\in\R_+$, on peut alors
définir les boules $B(a,r)$ et $B^\circ(a,r)$,
respectivement circonférenciées et non circonférenciées,
dans~$A$: ce sont les ensembles des $x\in A$ 
tels que $m(x-a)\leq r$, resp. $m(x-a)<r$.
La semi-distance~$m$ définit une topologie
sur~$A$ pour laquelle
les voisinages d'un point~$a\in A$
sont les ensembles qui contiennent une boule non circonférenciée
$B^\circ(a,r)$, avec $r>0$.

\begin{prop}
Soit $A$ un anneau commutatif et soit $m$ une valeur absolue sur~$A$;
munissons $A$ de la topologie définie par~$m$.
\begin{enumerate}
\item
L'addition et la multiplication de~$A$ sont continues,
\item 
Si $A$ est un corps, l'inversion est continue sur~$A^\times$.
\item 
L'application~$m$ est uniformément continue.
\item 
La topologie de~$A$ est séparée si et seulement si $m$ ne s'annule qu'en~$0$.
\end{enumerate}
\end{prop}
\begin{proof}
La continuité de l'addition résulte de l'inégalité triangulaire:
pour $a,b,x,y\in A$, on écrit
\[ m((x+y)-(a+b))=m((x-a)+(y-b))\leq m(x-a)+m(y-b), \]
si bien que $x+y$ tend vers~$a+b$ lorsque $x\to a$ et $y\to b$.
Pour la multiplication, on écrit de même
\begin{align*}
 m(xy-ab) & =m((x-a)(y-b)+a(y-b)+(x-a)b) \\
 &\leq m(x-a)m(y-b)+m(a)m(y-b)+m(x-a)m(b) , \end{align*}
si bien que $xy$ tend vers~$ab$ lorsque $x\to a$ et $y\to b$.

Supposons maintenant que $A$ soit un corps
et démontrons que l'inversion est continue.
Remarquons que si $m(a)\leq 1/2$, alors $m(1/(1-a))\leq 2$.
En effet, on écrit
$ \frac 1{1-a}=1+ \frac{a}{1-a}$, donc
\[ m(1/(1-a))\leq 1 + m(a) m(1/(1-a))\leq 1+ \frac12 m(1/(1-a)) , \]
d'où on déduit l'inégalité voulue.
Alors, pour $a,x\in A^*$, on écrit
\[ \frac1x-\frac1a = \frac{a-x}{ax}=\frac{a-x}{a^2} \frac1{1-(a-x)/a}.\]
Si $m(a-x)\leq \frac12 m(a)$, alors $m((a-x)/a))\leq 1/2$,
puis
\[ m(1/x-1/a) \leq 2 m(a-x)/m(a^2), \] 
d'où la continuité requise.

L'uniforme continuité de~$m$ résulte de la variante
\[ \abs{m(a)-m(b)} \leq m(a-b) \]
de l'inégalité triangulaire: comme $a=(a-b)+b$, on a $m(a)\leq m(a-b)+m(b)$,
d'où $m(a)-m(b)\leq m(a-b)$; par symétrie, $m(b)-m(a)\leq m(b-a)=m(a-b)$.

Si $m$ ne s'annule qu'en~$0$, l'application $(x,y)\mapsto m(x-y)$
est une distance, donc la topologie de~$A$ est séparée. 
Inversement, un élément $a\in A$ tel que $m(a)=0$
appartient à tout voisinage de~$0$, donc est nul
si la topologie de~$A$ est séparée. 
\end{proof}

\begin{exem}[Valeur absolue triviale] \label{exem.va-triviale}
Sur n'importe quel corps commutatif~$F$, on peut également
observer que l'on définit également une valeur
absolue~$m$ en posant
$m(0)=0$ et $m(a)=1$ pour tout $a\in F$
tel que $a\neq 0$. Cette valeur absolue
est dite \emph{triviale}; 
d'ailleurs, la définition originelle d'une valeur absolue
par  \cite{Kurschak-1913} l'excluait.

Pour la valeur absolue triviale,
la boule non circonférenciée de centre~$0$ et de rayon~$1$
ne contient que~$0$.
La topologie de~$F$ est donc \emph{discrète}.

Inversement, supposons que la topologie de~$F$ définie
par une valeur absolue~$m$ soit discrète,
et démontrons que $m$ est la valeur absolue triviale.
Par hypothèse, il existe un nombre réel~$r>0$ 
tel que la boule ouverte $B(0,r)$ ne contienne que le point~$0$. 
Soit $x\in F^\times$ et démontrons que $m(x)=1$.
Quitte à changer~$x$ en~$1/x$, on peut supposer que $m(x)<1$.
Il existe alors un entier~$n\geq 1$ tel que $m(x)^n<r$;
puisque $m(x^n)=m(x^n)$, on a donc $m(x^n)<r$, d'où $x^n=0$ d'après
le choix de~$r$. Ainsi, $x=0$.
\end{exem}

\begin{exem}[Valeurs absolues $p$-adiques]
\label{exem.vp}
Considérons le cas du corps~$\Q$ des nombres rationnels.
Soit $p$ un nombre premier. La valeur absolue $p$-adique~$m_p$
que nous allons définir
est essentiellement équivalente à la \emph{valuation $p$-adique}~$v_p$
qui est peut-être plus classique.

Si $a$ est un nombre rationnel non nul, 
il existe un unique entier relatif~$n$
telle que $a$ puisse s'écrire
$a=p^n b/c$, où $b$ et~$c$ sont des entiers non multiples de~$p$.
En effet, si $a=p^nb/c=p^{n'}b'/c'$,
alors $p^nbc'=p^{n'}b'c$, et l'égalité $n=n'$ découle
de l'unicité de la décomposition en facteurs premiers.
Cet entier~$n$ s'appelle la valuation $p$-adique de~$a$;
on le note~$v_p(a)$. Posons alors $m_p(a)=p^{-v_p(a)}$.

On pose enfin $m_p(0)=0$ et $v_p(0)=+\infty$.

Vérifions que $m_p$ est une valeur absolue sur~$\Z$;
on l'appelle la \emph{valeur absolue $p$-adique}.
La propriété~(\ref{defi.va}/\ref{va-01}) est évidente
et la propriété~(\ref{defi.va}/\ref{va-mul}) découle
de la définition.
La propriété~(\ref{defi.va}/\ref{va-it}) est en fait vraie sous une
forme plus forte:
\begin{enumerate}
\setcounter{enumi}3
\item \label{va-iu}
Pour $a,b\in \Q$, on a $m_p(a+b)\leq \sup(m_p(a),m_p(b))$
(\emph{inégalité ultramétrique}).
\end{enumerate}
Il suffit de traiter le cas où $a$ et~$b$ sont non nuls;
on écrit $a=p^n u/v$, $b=p^m r/s$,
où $u,v,r,s$ sont des entiers non multiples de~$p$.
Si $m\leq n$, alors $a+b=p^m (p^{n-m}u/v + r/s)
= p^m (p^{n-m}us+rv)/sv$; si $m'$ est l'exposant
de~$p$ dans la décomposition en facteurs premiers
de $p^{n-m}us+rv$, on a donc $v_p(a+b)=m+m'$,
et en particulier $v_p(a+b)\geq m=v_p(b)=\inf(v_p(a),v_p(b))$,
d'où l'inégalité voulue. Le cas où $n\leq m$ est symétrique.

Remarquons tout de suite la \emph{formule du produit:}
\[ m_\infty(a)\prod_{\text{$p$ premier}} m_p(a) = m_0(a) \]
qui relie valeur absolue usuelle~$m_\infty$, les valeurs
absolues $p$-adiques~$m_p$, où $p$ parcourt l'ensemble
des nombres premiers, et la valeur absolue triviale~$m_0$ sur~$\Q$.
L'égalité est évidente si $a=0$. Si $a\neq 0$, il s'agit
de démontrer 
\[ m_\infty(a) \prod_{\text{$p$ premier}} m_p(a) = 1\ ;\]
on écrit pour cela la décomposition en facteurs premiers,
$a=\pm \prod_p p^{v_p(a)}$, si bien que
\[ m_\infty(a) \prod_{p} m_p(a) = \left(\prod_p p^{v_p(a)}\right)
 \prod_p p^{-v_p(a)} = 1=m_0(a). \]
\end{exem}

\begin{exem}
Le cas où $F$ est le corps des fractions rationnelles
en une indéterminée~$T$ à coefficients
dans un corps commutatif~$k$ est particulièrement important.
Pour tout polynôme irréductible~$\pi\in k[T]$, 
on dispose d'une \emph{valuation} $\pi$-adique~$v_\pi$ sur~$k(T)$,
telle que 
$v_\pi(f)=n$ si $f\in k(T)$ s'écrit sous la forme
$\pi^n g/h$, où $g$ et~$h$ sont des polynômes non divisibles par~$\pi$.
On en déduit une \emph{valeur absolue} $\pi$-adique
sur~$k(T)$, définie par $m_\pi(f)=e^{- v_\pi(f)\deg(\pi)}$,
où la normalisation choisie ($c=e^{-\deg(\pi)}$)
anticipe la formule du produit.

On dispose également d'une valuation « à l'infini » $v_\infty$,
définie par $v_\infty(g/h)=-\deg(g)+\deg(h)$ si $g$ et $h$ 
sont des polynômes non nuls, et $v_\infty(0)=+\infty$.
Enfin, on pose $m_\infty(f)=e^{-v_\infty(f)}$ pour tout $f\in k(T)$.

Ces normalisations donnent lieu à la formule du produit:
pour toute fraction rationnelle~$f\in k(T)$, on a 
l'égalité 
\[ m_\infty(f) \cdot \prod_{\pi} m_\pi(f) = m_0(f) \]
qui traduit simplement l'égalité
\[ \sum_{\pi} v_\pi(f) \deg(\pi)  = \deg(f) \]
si $f$ est un polynôme non nul.
\end{exem}

\begin{exem}
L'exemple des valeurs absolues $p$-adiques sur~$\Q$
et celui des valeurs absolues sur un corps de fractions rationnelles
en une indéterminée associées à un polynôme irréductible
se généralisent \emph{mutatis mutandis}
au cas du corps des fractions~$F$ d'un anneau factoriel~$A$,
en remplaçant le nombre premier~$p$ ou le polynôme
irréductible~$\pi$  par un élément irréductible~$\pi$ de~$A$.

On commence en effet par définir la valuation $\pi$-adique d'un
élément $f\in F$ par $v_\pi(f)=+\infty$ si $f=0$,
et par $v_\pi(f)=n$ si $n$ est un entier tel qu'il existe
$u$ et~$v$ dans~$A$, non multiples de~$\pi$, tels que $f=\pi^n u/v$.
Pour tout nombre réel~$c>1$, on en déduit une \emph{valeur absolue} 
$\pi$-adique
sur~$F$, définie par $m_\pi(f)=c^{-v_\pi(f)}$.
Elle est ultramétrique.

Il n'y a pas, en général, de façon privilégiée de choisir~$c$.
De toutes façons, lorsque $c$ varie, ces valeurs absolues
définissent toutes la même topologie.
\end{exem}

\begin{lemm}
Soit $m$ une valeur absolue sur un anneau commutatif~$A$.
Les propriétés suivantes sont équivalentes:
\begin{enumerate}
\item La valeur absolue~$m$ est ultramétrique;
\item Pour tout $a\in A$ tel que $m(a)\leq 1$, on a 
 $m(1+a)\leq 1$ ;
\item Pour tout $n\in\N$, on a $m(n)\leq 1$.
\item La suite $(m(n))$ ne tend pas vers l'infini.
\end{enumerate}
\end{lemm}
\begin{proof}
L'implication~(1)$\Rightarrow$(2) est évidente:
si $m$ est ultramétrique et si $a\in A$
vérifie $m(a)\leq 1$, alors $m(1+a)\leq \sup(m(1),m(a))=1$.

L'implication (2)$\Rightarrow$(3) se démontre par une
récurrence immédiate.

Prouvons l'implication (3)$\Rightarrow$(1).
Soit donc $a,b\in A$. 
Soit $n$ un entier naturel non nul;
la formule du binôme pour $(a+b)^n$
et l'inégalité triangulaire entraînent
\[ m(a+b)^n = m((a+b)^n) \leq \sum_{k=0}^n m \big(\tbinom nk a^{n-k}b^k\big)
 \leq \sum_{k=0}^n m(a)^{n-k} m(b)^k \]
puisque $m\big(\binom nk\big)\leq 1$ pour tout $k$.
On en déduit l'inégalité 
\[ m(a+b) ^n \leq (n+1) \sup(m(a),m(b))^n. \]
En prenant les racines $n$-ièmes et en faisant tendre~$n$ vers l'infini, 
on en déduit l'inégalité ultramétrique.
\end{proof}

\begin{rema}\label{rema.car-zero}
Soit $F$ un corps muni d'une valeur absolue~$m$ qui n'est pas ultramétrique.
Alors, $F$ est de caractéristique zéro.
En effet, il existerait sinon un entier~$n$ tel que $m(n)$ soit maximal; 
en particulier, $m(n^2)\leq m(n)$, d'où $m(n)\leq 1$ puisque $m(n)\neq0$.
\end{rema}

\begin{prop}
Soit $m$ et $m'$ des valeurs absolues sur un corps commutatif~$F$.
Si la topologie définie par~$m'$ est plus fine que la topologie
définie par~$m$, il existe un nombre réel~$\rho\geq0$
tel que $m'(x)=m(x)^\rho$ pour tout $x\in F^\times$, et inversement.
\end{prop}
\begin{proof}
Soit $B$ l'ensemble des $x\in F$ tels que $m(x)\leq 1$;
c'est un voisinage de~$0$ pour la topologie définie par~$m$,
donc c'est aussi un voisinage de~$0$ pour la topologie
définie par~$m'$. Il existe donc un nombre réel~$r$
tel que $0<r<1$ et 
tel que pour tout $x\in F$ tel que $m'(x) \leq r$, on ait $m(x)\leq 1$.

Supposons que $0$ soit le seul élément de~$F$
tel que $m'(x)\leq r$; alors la topologie de~$F$ définie
par~$m'$ est la topologie discrète, donc $m'$ est la valeur
absolue triviale de~$F$ (exemple~\ref{exem.va-triviale})
et l'on a $m'(x)=m(x)^0$ pour tout $x\in F$.

Il existe sinon un élément $a\in F$ tel que $0<m'(a)\leq r$;
posons $\rho=\log(m(a))/\log(m'(a))$.
Soit $x$ un élément de~$F$ tel que $m'(x)<1$.
Pour tout entier~$n\geq 1$, il existe un unique
entier~$p\geq 1$ tel que $m'(a)^{p+1} < m'(x)^n \leq m'(a)^p$;
en fait, $p$ est la partie entière de $n \log(m'(x))/\log(m'(a))$.
Comme $m'(x^n/a^p)=m'(x)^n/m'(a)^p\leq 1$, 
on a $m'(a\cdot x^n/a^p)\leq m'(a)\leq r$,
donc $m(a\cdot x^n/a^p)\leq 1$
puis $m(x) \leq m(a)^{(p-1)/n}$. En faisant
tendre~$n$ vers l'infini, on obtient $m(x)\geq m(a)^{\log(m'(x))/\log(m'(a))}
=m'(x)^\rho$.
De même, comme $m'(a^{p+1}/x^n)=m'(a)^{p+1}/m'(x)^n<1$,
on a $m'(a\cdot a^{p+1}/x^n)\leq m'(a)\leq r$,
donc $m(a\cdot a^{p+1}/x^n)\leq 1$
puis $m(x)\geq m(a)^{(p+2)/n}$. En faisant tendre~$n$
vers l'infini, on obtient $m(x)\geq m'(x)^\rho$.
\end{proof}

\begin{exems}
\begin{enumerate}
\item La topologie discrète est plus fine que toute topologie,
et est donnée par la valeur absolue triviale.
Cela correspond au cas $\rho = 0$ dans la proposition.

\item
D'après cette proposition, les valeurs absolues~$m$ sur~$\R$
qui définissent une topologie plus fine que  la
topologie usuelle sont de la forme $m(x)=\abs x^\rho$,
où $\rho$ est un nombre réel~$\geq 0$.
Le cas $\rho=0$ correspond à la topologie discrète.
On laisse aux lecteurices le soin de vérifier 
que l'on a $\rho\leq 1$ et qu'inversement, toute telle valeur
correspond bien à une valeur absolue sur~$\R$.

Le cas du corps des nombres complexes est analogue.

On peut cependant considérer sur~$\R$ ou~$\C$ des valeurs absolues
très exotiques au sens où les topologies définies ne sont pas 
comparables à la topologie classique.
Un exemple fameux est donné par la démonstration
de~\citet{Monsky-1970} du théorème selon lequel dans
toute décomposition d'un carré en un nombre fini
triangles de même aire, le nombre de triangles de la décomposition est pair.
Dans cette preuve, Monsky introduit une valeur absolue $2$-adique sur~$\R$,
pour laquelle $m(2)=1/2$.

%
%

\item
Les topologies $p$-adiques sur~$\Q$
sont deux à deux incomparables, et elles sont incomparables
à la topologie définie par la valeur absolue usuelle.
Elles sont toutes moins fines que la topologie discrète
définie par la valeur absolue triviale.
\end{enumerate}
\end{exems}

\begin{theo}[\citealp{Ostrowski-1916}]
\begin{enumerate}
\item
Soit $m$ une valeur absolue sur~$\Q$. Alors $m$
est la valeur absolue triviale, une puissance de la valeur absolue
usuelle, ou une puissance d'une valeur absolue $p$-adique.
\item
Soit $k$ un corps et soit $m$ une valeur absolue sur~$k(T)$
dont la restriction à~$k$ est la valeur absolue triviale.
Alors $m$ est la valeur absolue triviale, 
une puissance de la valeur absolue associée au degré, 
ou une puissance d'une valeur absolue associée à un polynôme irréductible.
\end{enumerate}
\end{theo}
\begin{proof}
\begin{enumerate}
\item
Soit $a$ un entier naturel tel que $a > 1$.
Soit $x$ un entier $\geq1$; soit $n$ un entier~$\geq 1$ et
soit $p$ le plus petit entier tel que $x^n< a^{p}$
et soit $x^n=x_0+ax_1+\dots+a^{p-1} x_{p-1}$ la décomposition 
en base~$a$ de~$x^n$, où $0\leq x_0,\dots,x_{p-1}\leq a-1$.
On en déduit $m(x^n)\leq \sup(m(1),\dots,m(a-1)) p m(a)^p $,
d'où
\[ m(x) \leq \sup(m(1),\dots,m(a-1))^{1/n} p^{1/n} \sup(1,m(a))^{p/n}. \]
En faisant tendre~$n$ vers l'infini, $p/n$ tend vers~$\log(x)/\log(a)$
et on obtient
l'inégalité $m(x)\leq \sup(1,m(a))^{\log(x)/\log(a)}$, d'où
$m(x)^{1/\log(x)}\leq \sup(1,m(a))^{1/\log(a)}$. 

S'il existe $x\geq 1$ tel que $m(x)>1$, on en déduit que $m(a)>1$
pour tout entier~$a>1$, puis que $m(x)^{1/\log(x)}\leq m(a)^{1/\log(a)}$
pour tous $a,x>1$. En notant~$\rho$ cette valeur constante,
on a donc $m(x)=x^\rho$ pour tout entier~$x>1$, puis
$m(x)=\abs x^\rho$ pour tout $x\in\Q$.

Sinon, on a $m(x)\leq 1$ pour tout entier~$x\geq 2$:
la valeur absolue~$m$ est ultramétrique.

Si $m(n)=1$ pour tout entier non nul~$n$,
alors $m(a)=1$ pour tout nombre rationnel non nul~$a$,
et $m$ est la valeur absolue triviale.

Sinon, soit~$p$ le plus petit entier~$\geq 1$ tel que $m(p)<1$.
On a $p> 1$; si $u$ et $v$ sont des entiers~$>1$
tels que $p=uv$, alors $m(u)=m(v)=1$ par minimalité de~$p$,
donc $m(p)=1$. Ainsi, $p$ est un nombre premier.
Soit $u$ un entier tel que $m(u)<1$ et soit $u=kp+r$
sa division euclidienne par~$p$; si $r\neq 0$, on a $m(r)=1$
par minimalité de~$p$, donc $m(p)\geq m(k)m(p)=m(kp)=m(r-u)=1$,
et cela contredit la définition de~$p$. Ainsi, $r=0$
et $u$ est multiple de~$p$. On en déduit aisément
que $m(a)=m(p)^{v_p(a)}=m_p(a)^{\rho}$, où $\rho=-\log(m(p))/\log(p)$.

\item
Dans ce cas, $m$ est ultramétrique.

Si $m(T)\leq 1$, alors $m(f)\leq 1$ pour tout polynôme $f\in k[T]$.
Soit $I$ l'ensemble des polynômes $f\in k[T]$ tels que $m(f)<1$;
c'est un idéal de~$k[T]$, et c'est même un idéal \emph{premier}.
Si $I=(0)$, alors $m(f)=1$ pour tout polynôme non nul~$f$,
d'où l'on déduit que $m$ est la valeur absolue triviale sur~$k(T)$.
Sinon, il existe un polynôme irréductible~$\pi$ tel que $I=(\pi)$,
et on en déduit que $m$ est une puissance de la valeur absolue~$m_\pi$.

Il reste à traiter le cas où $m(T)>1$.
Dans ce cas, on a $m(aT^n)=m(T)^n$  pour tout $a\in k^\times$,
et l'inégalité ultramétrique entraîne que $m(f)=m(T)^{\deg(f)}$.
En posant $\rho=\log(m(T))$, on obtient $m(f)=m_\infty(T)^\rho$.
\qedhere
\end{enumerate}
\end{proof}
\begin{exem}\label{exem.gauss}
Soit $F$ un corps muni d'une valeur absolue ultramétrique~$m$.
On étend~$m$ à l'anneau $F[T]$ en posant
\[ m(f) = \sup(m(a_n)),  \qquad f = \sum_n a_n T^n. \]
On vérifie directement que l'application ainsi définie vérifie les
axiomes~(\ref{defi.va}/\ref{va-01}) et (\ref{defi.va}/\ref{va-it})
d'une valeur absolue;
elle vérifie même l'inégalité ultramétrique~(\ref{exem.vp}/\ref{va-iu}).

Moins évidente est la propriété~(\ref{defi.va}/\ref{va-mul}),
c'est-à-dire que l'on a $m(fg)=m(f)m(g)$ pour tous~$f,g\in F[T]$.
La propriété est évidente si $f=0$ ou $g=0$; supposons
donc que $f\neq0$ et $g\neq 0$.
Soit $A=F^\circ$ l'anneau de valuation de~$F$,
soit $P$ son idéal maximal~$F^{\circ\circ}$ et soit $k=A/P$
son corps résiduel.
Posons $f=\sum a_n T^n$, $g=\sum b_nT^n$, $h=\sum c_n T^n$;
soit $r$ un entier tel que $m(f)=a_r$ et
soit $s$ un entier tel que $m(g)=b_s$.
on a $f=a_r (f/a_r)$, $g=b_s (g/b_s)$, et 
la définition entraîne 
\begin{align*}
m(fg) & = m(a_r) m(b_s) m((f/a_r)(g/b_s)),  \\
m(f)m(g) & = m(a_r) m(b_s) m(f/a_r) m(g/b_s).
\end{align*}
Pour démontrer l'égalité voulue, on peut ainsi supposer
que $m(f)=m(g)=1$, et il s'agit de démontrer que $m(h)=1$ ; 
l'hypothèse  signifie que $f$ et $g$ appartiennent à~$A[T]$
et que leurs images~$\overline f$ et $\overline g$
dans $k[T]$ sont non nulles.
Comme l'anneau~$k[T]$ est intègre, le produit~$\overline f \overline g$
n'est pas nul; puisque ce produit est l'image du polynôme~$fg\in A[T]$,
on a $m(c_n)\leq 1$ pour tout~$n$,
et il existe un entier~$t$ tel que $m(c_t)=1$.
Ainsi, $m(h)=1$.

Les lecteurices auront peut-être reconnu quelques
cas particuliers de cet énoncé de multiplicativité:
lorsque $F$ est le corps des fractions d'un anneau factoriel~$R$
et que la valeur absolue~$m$ est associée à un élément irréductible,
il est souvent connu comme le « lemme de Gauss »:
il apparaît comme une étape cruciale.
dans la démonstration du théorème de Gauss
que l'anneau~$R[T]$ est factoriel.

Par construction, $m(f)\neq 0$ pour $f\neq 0$;
l'application~$m$ se prolonge alors de manière unique en une valeur
absolue sur le corps~$F(T)$ des fractions rationnelles ---
on l'appelle la \emph{valeur absolue de Gauss}.
\end{exem}

\section{Complétude, locale compacité. Normes}

\subsection{}
Soit $A$ un anneau commutatif et soit $m$ une valeur absolue sur~$A$.
Le séparé--complété~$\widehat A$ de~$A$ pour la topologie
définie par~$m$
est le quotient de l'ensemble des suites de Cauchy dans~$A$ par
la relation d'équivalence pour laquelle $(a_n)\sim (b_n)$
si et seulement si $m(a_n-b_n)\to 0$.
Les suites de Cauchy forment un sous-anneau de
l'ensemble des suites, et la relation d'équivalence en question
est celle associée à l'idéal formé des suites qui tendent vers~$0$.
Ainsi, $\widehat A$ possède une structure d'anneau pour laquelle
l'application canonique $A \to \widehat A$ est un morphisme d'anneau.
Le noyau de ce morphisme est le noyau~$P$ de la valeur absolue~$m$;
il est en particulier injectif lorsque~$A$ est un corps.

Comme l'application~$m\colon A\to\R_+$ 
est uniformément continue (et passe au quotient par l'idéal~$P$),
l'application canonique s'étend par densité en une unique
application continue de~$\widehat A$ dans~$\R_+$; 
c'est encore une valeur absolue. Concrètement
si $\alpha\in\widehat A$ est la classe d'une suite de Cauchy~$(a_n)$,
alors la suite~$m(a_n)$ converge et~$m(\alpha)$ est sa limite.
Les axiomes d'une valeur absolue se vérifient immédiatement.

\begin{lemm}
Soit $F$ un corps et soit $m$ une valeur absolue sur~$F$.
Alors $\widehat F$ est un corps.
\end{lemm}
\begin{proof}
Soit $\alpha$ un élément non nul de~$\widehat F$ et soit $(a_n)$
une suite de Cauchy qui représente~$\alpha$. 
Comme $(a_n)$ ne tend pas vers~$0$ dans~$F$, la suite $(m(a_n))$ 
ne tend pas vers~$0$ dans~$\R_+$, si bien que $m(\alpha)>0$.
En particulier, il existe un nombre réel~$r>0$
tel que $m(a_n)\geq r$ pour tout entier~$n$ assez grand.
Pour tout entier~$n$, posons $b_n=1/a_n$ si $a_n\neq0$,
et $b_n=0$ sinon. 
Pour $n,p$ assez grands, on a $ b_n-b_p= (a_p-a_n)/a_pa_n$, 
ce qui entraîne que la suite~$(b_n)$ est de Cauchy;
notons $\beta$ sa classe dans~$\widehat F$. 
La suite $(a_nb_n)$ est stationnaire de limite~$1$,
on a donc $\alpha\beta=1$.
\end{proof}

\subsection{}\label{ss.archimedien}
Soit $F$ un corps muni d'une valeur absolue~$m$ qui n'est pas ultramétrique;
on sait que $F$ est de caractéristique zéro (remarque~\ref{rema.car-zero}) et
la restriction de~$m$ à son sous-corps premier~$\Q$ est 
une valeur absolue sur~$\Q$ qui n'est pas ultramétrique.
D'après le théorème d'Ostrowski, il existe
un nombre réel~$\rho>0$ tel que $m(a)=\abs a^\rho$ pour tout $a\in\Q$.
En particulier, la suite $(m(n))$ tend vers l'infini lorsque
l'entier~$n$ tend vers l'infini.
Ou encore, pour tout $x\in F$ tel que $m(x)\neq 0$
et tout nombre réel~$M$,
il existe un entier~$n$ tel que $m(nx)\geq M$.
Cette propriété, reminiscente de la propriété d'Archimède pour
le corps~$\R$ des nombres réels,
conduit à dire que $m$ est une \emph{valeur absolue archimédienne.}

De plus, la structure de $\Q$-algèbre sur le complété~$\widehat F$ de~$F$
se prolonge en une structure de $\R$-algèbre.

Si $\widehat F$ contient un élément~$i$ tel que $i^2=-1$,
un théorème d'Ostrowski affirme que $\widehat F$ est isomorphe à~$\C$,
et que sa valeur absolue  est une puissance de la valeur absolue usuelle.
Dans le cas contraire,
on démontre que la valeur absolue~$m$ de~$\widehat F$
s'étend en une valeur absolue sur le corps~$\widehat F(i)$,
lequel est donc isomorphe à~$\C$ d'après le premier cas.

Cela montre que la théorie des corps munis d'une valeur absolue
n'apporte de nouveauté par rapport à l'analyse classique
que dans le cas ultramétrique.

\subsection{}
Soit $A$ un anneau commutatif et soit $m$ une valeur absolue~$m$
ultramétrique sur~$A$. 
L'inégalité ultramétrique donne lieu à des
des comportements topologiques/géométriques assez différents de 
ceux de l'analyse réelle.

\begin{enumerate}
\item
\emph{Pour tous $a,b\in A$ tel que $m(b)<m(a)$,
on a $m(a+b)=m(a)$.}

Par l'inégalité ultramétrique, on a déjà $m(a+b)\leq \sup(m(a),m(b))
=m(a)$.
En écrivant $a=(a+b)-b$, l'inégalité ultramétrique
fournit également $m(a)\leq \sup(m(a+b),m(-b))$. 
Si $m(a+b)\geq m(-b)$, on obtient l'inégalité $m(a)\leq m(a+b)$;
l'autre cas est absurde puisqu'il entraîne $m(a)\leq m(-b)=m(b)$.

\item
\emph{Si deux boules se rencontrent, celle de plus
petit rayon est contenue dans l'autre. En particulier,
tout point d'une boule en est un centre.}
Soit en effet $a,b\in A$ et $r,s\in\R_+$, 
et soit $c$ un point commun aux boules $B(a,r)$ et $B(b,s)$.
Alors $m(c-a)\leq r$ et $m(c-b)\leq s$; si $r\leq s$,
on a donc $m(a-b)=m((c-b)-(c-a))\leq \sup(r,s)=s$, 
donc $a\in B(b,s)$. Par suite, si $x\in B(a,r)$,
on a $m(x-b)=m((x-a)-(a-b))\leq\sup(m(x-a),m(a-b))\leq s$,
si bien que $B(a,r)\subseteq B(b,s)$.

\item
\emph{Les boules circonférenciées sont à la fois ouvertes et fermées.}
La boule $B(a,r)$ est ouverte, 
car si $x\in B(a,r)$, on a $B(x,r)\subseteq B(a,r)$ (en fait égalité). 
Elle est fermée car si $x\notin B(a,r)$, 
on a $B(x,r)\cap B(a,r)=\emptyset$ (sinon, on aurait
$B(x,r)\subseteq B(a,r)$, contredisant l'hypothèse $x\notin B(a,r)$).

Tout point de~$A$ possède une base de voisinages
ouverts et fermés, à savoir les boules circonférenciées.
Par suite, \emph{la topologie de~$A$ est totalement discontinue.}

\item
\emph{Une suite~$(a_n)$
est de Cauchy si et seulement si la suite des différences
$(a_{n+1}-a_n)$ tend vers~$0$.} En effet, si cette dernière
propriété est satisfaite, l'égalité
\[  a_{n+k}-a_n  = \sum_{i=0}^{k-1} (a_{n+i+1}-a_{n+i}) \]
entraîne 
\[ m(a_{n+k}-a_n)\leq \sup_{0\leq i<k} m(a_{n+i+1}-a_{n+i}), \]
d'où l'on déduit que la suite~$(a_n)$ est de Cauchy.

Plus généralement, un monde ultramétrique est un monde dans lequel
la politique des petits pas ne permet pas d'aller bien loin.
\end{enumerate}

\subsection{}
Soit $F$ un corps muni d'une valeur absolue ultramétrique~$m$.
L'ensemble $A$ des éléments~$a\in F$
tels que $m(a)\leq 1$ est un sous-anneau de~$F$;
par construction, pour tout élément non nul~$a$ de~$F$,
on a $a\in A$ ou $1/a\in A$:
l'anneau~$A$ est ce qu'on appelle
un \emph{anneau de valuation} du corps~$F$.

Un élément $a\in A$ est inversible dans~$A$
si et seulement si $m(a)=1$:
en effet, supposant $a\neq0$,
il faut et il suffit que $1/a$ appartienne à~$A$,
c'est-à-dire $m(1/a)\leq 1$, d'où $m(a)\geq 1$.

Par suite, l'ensemble $M$ des éléments $a\in A$
tels que $m(a)<1$ est un idéal de~$A$
et c'est le seul idéal maximal de l'anneau~$A$.

Ainsi, et malgré la vocation analytique de l'introduction
de la notion de valeur absolue,
les valeurs absolues ultramétriques explorent 
un monde mathématique très proche de l'algèbre.

\subsection{}
La restriction à~$F^\times$ de l'application $x \mapsto -\log(m(x))$
est un morphisme de groupes de $F^\times$ dans~$\R$;
son image~$\Gamma$ est donc un sous-groupe de~$\R$ et
il y a trois possibilités pour l'image de ce morphisme:
\begin{itemize}
\item Elle est triviale: c'est le cas où  la valeur absolue~$m$
est la valeur absolue triviale, et l'anneau~$A$ coïncide avec
le corps~$F$;
\item C'est un sous-groupe isomorphe à~$\Z$;
il existe donc un élément $a\in A$ tel que $-\log(m(a))$
est l'unique générateur strictement positif de~$\Gamma$. Cet 
élément~$a$ est un générateur de l'idéal~$M$, qui est donc principal.
On peut en déduire que $A$ est un anneau principal
(si $I$ est un idéal non nul de~$A$, il existe un plus grand entier~$n$
tel que $-\log(m(x))\geq-n \log(m(a))$ pour tout $x\in A$,
et l'on a $I = \langle a^n\rangle$)
qui n'est pas un corps (car $a$ n'est pas inversible dans~$A$);
\item C'est un sous-groupe dense de~$\R$.
Dans ce cas, l'idéal~$M$ n'est pas principal,
il n'est même pas de type fini\footnote{Dans
un anneau de valuation comme l'anneau~$A$,
tout idéal de type fini est principal…}, de sorte que 
l'anneau~$A$ n'est pas noethérien.
\end{itemize}

On dit la \emph{valuation de~$F$ est discrète} si le groupe~$\Gamma$
est isomorphe à~$\Z$.
La terminologie est traditionnelle, mais malheureuse,
car elle n'a rien à voir avec la notion topologique
(qui correspond au cas de la valeur absolue triviale). 

\begin{prop}
Soit $F$ un corps muni d'une valeur absolue ultramétrique
non triviale.
Les propriétés suivantes sont équivalentes:
\begin{enumerate}
\item La topologie de~$F$ est localement compacte;
\item La topologie de l'anneau de valuation~$A$ est compacte;
\item La topologie de~$F$ est complète,
sa valuation est discrète et son corps résiduel $A/M$ est fini. 
\end{enumerate}
Si elles sont satisfaites, $A$ est homéomorphe à un ensemble de Cantor.
\end{prop}
\begin{proof}
(1)$\Rightarrow$(2). Supposons $F$ localement compact. Il existe alors
un nombre réel~$r>0$ tel que $B(0,r)$ est compact.
La fonction $a\mapsto m(a)$ sur~$B(0,r)$ est continue,
elle est donc majorée et atteint sa borne supérieure;
cela ramène au cas où il existe $a\in A$ tel que $r=m(a)$.
L'application $x\mapsto x/a$ de~$F$ dans lui-même est un homéomorphisme
et l'image de $B(0,r)$ est la boule $B(0,1)$. Ainsi, $A$ est compact.

(2)$\Rightarrow$(3). Supposons maintenant $A$ compact. 
Tout d'abord, l'espace métrique~$F$ est localement compact, donc complet.

Soit $p\colon A\to A/M$ l'application de réduction;
pour $a,b\in A$, les conditions $p(a)=p(b)$ et $m(a-b)<1$ 
sont donc équivalentes.
Pour tout $u\in A/M$, choisissons un élément $s(u)\in A$
tel que $p(s(u))=u$; alors $A$ est la réunion des boules non
circonférenciées $B^\circ(s(u),1)$, pour $u\in A/M$; elles sont ouvertes
et deux à deux disjointes; 
comme $A$ est compact, elles sont en nombre fini, ce qui prouve que $A/M$
est fini. 

La réunion des $B^\circ(s(u),1)$, pour $u\in A/M$ et $u\neq 0$
est une partie ouverte de~$A$; c'est le complémentaire de $B^\circ(s(0),1)=B^\circ(0,1)=M$ qui est donc une partie fermée; par suite, 
la fonction continue $a\mapsto m(a)$ sur $B^\circ(0,1)$ atteint
sa borne supérieure, ce qui prouve qu'il existe $a\in M$
tel que $m(x)\leq m(a)$ pour tout $x\in M$.
Observons que $a\neq0$ (sinon,  $M=0$, donc $A$ est un corps, donc $A=F$
et la valeur absolue de~$F$ est triviale); pour $x\in M$,
on a $m(x/a)\leq 1$, donc $x/a\in M$ et $x\in \langle a\rangle$;
ainsi, l'idéal~$M$ est principal et la valuation de~$F$ est discrète.

(3)$\Rightarrow$(1). Fixons un générateur~$a$ de l'idéal~$M$
et une application $u\mapsto s(u)$ de~$A/M$ dans~$A$
tel que $p(s(u))=u$ pour tout $u\in A/M$.
Pour toute suite $(u_n)$ de $(A/M)^\N$,
la série $\sum s(u_n) a^n$ dans~$A$ converge, car
son terme général tend vers~$0$, puisque $F$ est complet.
L'application
$\phi \colon (A/M)^\N\to A$ telle que $\phi(u)=\sum_n s(u_n) a^n$
est continue, lorsque $(A/M)^\N$ est muni de la topologie
produit des topologies discrètes.
En effet, si les $n$~premiers termes de~$u$ et~$v$ coïncident,
alors $\phi(u)-\phi(v)$ est multiple de~$a^n$, 
donc $m(\phi(u)-\phi(v))\leq m(a)^n$.
Si, de plus, si $u_n\neq v_n$, 
alors $m(\phi(u)-\phi(v))=m(a)^n$, ce qui prouve que $\phi$ est injective.

Démontrons que $\phi$ est surjective.
Soit $x\in A$; posons $x_0=x$ et $u_0=p(x)$; 
comme $x-s(u_0)\in M$, il existe un unique élément $x_1\in A$
tel que $x-s(u_0)=ax_1$. Par récurrence, on construit des suites
$(u_n)$  dans~$A/M$ et $(x_n)$ dans~$A$
telles que 
$x = s(u_0)+a s(u_1)+\dots+a^n s(u_n)+a^{n+1}x_{n+1}$
pour tout entier~$n$. On constate que $x=\phi(u)$, donc $\phi$ est surjective.

L'application $\phi\colon (A/M)^\N\to A$ est bijective, continue;
c'est donc un homéomorphisme car $(A/M)^\N$ est compact et $A$ est séparé.
Lorsque $A/M$ est de cardinal~$2$ (ce qui signifie que $F=\Q_2$
ou $F=\F_2\lpar T\rpar$, 
la caractérisation des points de l'ensemble triadique de Cantor~$K_3$ 
par leur développement triadique exhibe un homéomorphisme
de cet ensemble avec $\{0,1\}^\N$.
En général, l'espace topologique produit~$(A/M)^\N$, produit infini d'un
espace topologique discret fini (avec au moins deux éléments), est 
infini, métrisable, compact, séparable et totalement discontinu:
un résultat classique de topologie générale entraîne 
qu'il est homéomorphe à~$K_3$.
\end{proof}

\begin{remas}
\begin{enumerate}
\item
De la démonstration de~(3)$\Rightarrow$(1), on peut 
extraire l'énoncé: si $F$ est complet, 
sa valuation non triviale et discrète, 
alors l'anneau de valuation~$A$ est homéomorphe
à l'espace « prodiscret » $(A/M)^\N$.

\item
Lorsque la valeur absolue de~$F$ est triviale, 
la topologie de~$F$ est discrète,
donc localement compacte puisque pour tout $a\in F$,
le singleton~$\{a\}$  est un voisinage ouvert compacts  de~$a$.
En revanche, comme l'anneau~$A$ coïncide avec~$F$,
il n'est compact que si $F$ est fini.
\end{enumerate}
\end{remas}

\begin{exem}
Soit $k$ un corps. Munissons l'anneau~$A=k[T]$
de la valeur absolue~$m_T$ pour laquelle $m(T)=1/e$
et $m(a)=1$ pour tout $a\in k^\times$.
Le complété~$\widehat A$ de~$A$ s'identifie naturellement
à l'anneau $k\lbra T\rbra$ des séries formelles en~$T$.
\end{exem}

\begin{exem}
On appelle \emph{corps des nombres $p$-adiques}
le complété~$\Q_p$ du corps des nombres rationnels
pour la valeur absolue $p$-adique~$m_p$ (normalisée par $m_p(p)=1/p$).
Son anneau de valuation est noté~$\Z_p$;
c'est l'ensemble des entiers $p$-adiques.

La construction du séparé--complété 
entraîne que $\Z_p$ est l'adhérence dans~$\Q_p$ de l'anneau~$\Z$,
que son idéal maximal est l'adhérence de l'idéal~$\langle p\rangle$,
et que le corps résiduel de~$\Z_p$ s'identifie au corps fini~$\Z/p\Z$.
La démonstration de la proposition précédente
fournit ainsi un homéomorphisme de~$(\Z/p\Z)^\N$ sur~$\Z_p$.
Cet homéomorphisme dépend du choix d'éléments $s(u)$ de~$\Z_p$
pour $u\in\Z/p\Z$. S'il existe un choix meilleur que d'autres,
multiplicatif 
(voir l'exemple~\ref{exem.teichmuller}),
il est naturel de choisir $s(u)$ dans $\{0,\dots,p-1\}$.
C'est d'ailleurs ce que fait \cite{Hensel-1904b} lorsqu'il introduit
les nombres $p$-adiques, comme des sortes de nombres entiers
donnés par un développement illimité en base $p$.
En revanche, si nous serions tentés d'écrire $ \dots s_n s_{n-1}\dots s_0 $ 
un tel nombre, Hensel l'écrit  $s_0s_1s_2\dots$!
Dans cet article, Hensel ne donne que peu d'indication
sur l'origine de ses idées, mais on peut en deviner une trace
dans son article~\citep{Hensel-1901} où il élabore explicitement
une analogie entre le développement en base~$p$ des nombres entiers
(et même des nombres algébriques)
et le développement en série de puissances d'une fonction algébrique. 

On peut préférer une présentation plus algébrique des nombres $p$-adiques.
Pour tout entier~$n$, l'application de~$\Z$ dans~$\Z_p$
induit un isomorphisme d'anneaux de~$\Z/p^n\Z$ sur~$\Z_p/p^n\Z_p$.
Les inverses de ces isomorphismes fournissent un homomorphisme d'anneaux
\[ \Z_p \to \prod_n \Z/p^n\Z \]
qui est injectif, et dont l'image
(la limite projective de ce système d'anneaux,
notée $\varprojlim_n\Z/p^n\Z$), 
est le sous-anneau de $\prod_n \Z/p^n\Z$ formé des suites $(a_n)$ 
telles que $a_{n+1}\equiv a_n \pmod{p^n}$. De plus, 
cet homomorphisme
induit un homéomorphisme de~$\Z_p$ sur son image
lorsque $\prod_n \Z/p^n\Z$ 
est muni de la topologie produit des topologies discrètes.
\end{exem}

\begin{defi}
Soit $F$ un corps muni d'une valeur absolue~$\abs{\,\cdot\,}$
et soit $V$ un $F$-espace vectoriel.
Une fonction $\norm{\,\cdot\,}\colon V\to \R_+$ est
appelée \emph{norme} si
\begin{enumerate}
\item Pour $v\in V$, on a $\norm v=0$ si et seulement si $v=0$;
\item On a $\norm {av}=\abs a \norm v$ pour tout $a\in F$
et tout $v\in F$;
\item On a $\norm{v+w}\leq \norm v+\norm w$ pour tous $v,w\in V$.
\end{enumerate}
\end{defi}

\subsection{}
Une norme sur~$V$ définit une distance sur~$V$, donc une topologie
pour laquelle l'addition de~$V$ et la multiplication sont continues.

D'après la proposition suivante,
deux normes $\norm\cdot_1$ et $\norm\cdot_2$ 
sur~$V$ définissent la même topologie
sur~$V$ si et seulement si elles sont équivalentes,
c'est-à-dire s'il existe des nombres réels~$c,c'>0$
tels que $c' \norm v_2\leq \norm v_1\leq c \norm v_2$  pour tout $v\in V$.

\begin{prop}
Soit $F$ un corps muni d'une valeur absolue non triviale.
soit $V$ et $ W$ des $F$-espaces vectoriels  normés
et soit $f\colon V\to W$ une application linéaire.
Les conditions suivantes sont équivalentes:
\begin{enumerate}
\item L'application~$f$ est continue;
\item L'application~$f$ est continue en~$0$;
\item Il existe un nombre réel~$c$ tel que
$\norm {f(v)}\leq c \norm v$ pour tout $v\in V$.
\end{enumerate}
\end{prop}
\begin{proof}
L'implication (1)$\Rightarrow$(2) est évidente,
et l'hypothèse~(3) entraîne que $f$ est $c$-lipschitzienne, donc continue.

Supposons~(2) et démontrons~(3); c'est là que nous devrons
utiliser l'hypothèse que la valeur absolue de~$F$ n'est pas triviale.

La continuité de~$f$ entraîne qu'il existe
un nombre réel~$r>0$ tel que si $\norm v<r$, alors $\norm{f(v)}\leq 1$.
Puisque la valeur absolue de~$F$ n'est pas triviale,
il existe un élément~$a\in F$ tel que $0<\abs a <1$;
nous allons démontrer que 
$\norm{f(v)}\leq (1/\abs a r) \norm v$.
En effet, soit $v\in V$; l'inégalité est évidente si $v=0$;
supposons donc $v\neq 0$
et soit $n$ le plus petit entier relatif tel que $\norm{ a^n v}<r$.
On a donc $\abs{a}^n \norm{f(v)}=\norm{f(a^n v)}\leq 1$;
d'autre part, la définition de~$n$ entraîne
que $\norm{a^{n-1}v}\geq r$, donc $1\leq (\abs{a}^{n-1}/r) \norm v$.
Par suite, $\norm{f(v)}\leq (1/\abs{a}r)  \norm{v}$,
comme il fallait démontrer.
\end{proof}

\begin{theo}
Soit $F$ un corps muni d'une valeur absolue pour laquelle il est complet.

\begin{enumerate}
\item
Soit $V, V'$ des $F$-espaces vectoriels normés
et soit $f\colon V\to V'$ une application linéaire.
Si $V$ est de dimension finie, alors $f$ est lipschitzienne,
et en particulier continue.

\item
Soit $V$ un $F$-espace vectoriel de dimension finie.
Toutes les normes sur~$V$ sont équivalentes;
la topologie qu'elles induisent sur~$V$ est complète.

\item 
Soit $V$ un $F$-espace vectoriel normé
et soit $W$ un sous-espace vectoriel de~$V$.
Si $W$ est de dimension finie, alors $W$ est fermé.
\end{enumerate}
\end{theo}
\begin{proof}
On raisonne le théorème par récurrence sur la dimension de~$V$.
Soit $(e_1,\dots,e_n)$ une base de~$V$.

(1) Si $f$ n'est pas continue, il existe pour tout~$k\geq 1$
un vecteur non nul $v^{(k)}\in V$ tel que $\norm {v^{(k)}} < k^{-1} \norm{f({v^{(k)})}}$. Écrivons $v^{(k)}=a_1e_1+\dots+a_ne_n$, où $a_1,\dots, a_n\in F$;
par homogénéité, 
on peut supposer que l'une des coordonnées~$a_m$ est égale à~$1$
et toutes les autres sont de valeur absolue~$\leq 1$.
Puis, quitte à considérer une sous-suite et à réordonner la base~$(e_1,\dots,e_n)$, on peut supposer que cette coordonnée est celle d'indice~$n$;
posons $W=\langle e_1,\dots,e_{n-1}\rangle$.
Par hypothèse de récurrence, l'espace~$W$ est fermé dans~$V$.

Observons que $\norm{f(v^{(k)})}\leq \sum_{j=1}^{n}\norm{f(e_j)}$
pour tout~$k$,
si bien que $\norm{v^{(k)}} \to 0$.
Cela entraîne que la suite~$(v^{(k)}-e_n)$ tend vers~$-e_n$;
comme elle est contenue dans~$W$, cela prouve que $-e_n\in W$,
ce qui est absurde.

(2) 
En appliquant~(1) à l'application identique de~$V$
pour deux normes arbitraires sur~$V$,
on obtient une des deux inégalités requises par l'équivalence
des normes, et l'autre s'en déduit par symétrie.

Par ailleurs, le choix de la base~$(e_1,\dots,e_n)$
induit une norme $\norm{\cdot}$ sur~$V$ telle
que $\norm{\sum a_i e_i}=\sup(\abs{a_1},\dots,\abs{a_n})$
pour tout $(a_1,\dots,a_n)\in F^n$.
Il est facile de voir que cette norme rend~$V$ complet.

(3)  La norme de~$V$ définit une norme sur~$W$
qui le rend complet, d'après~(2). Comme un sous-espace complet
d'un espace métrique est fermé, cela entraîne que $W$ est fermé.
\end{proof}

\section{Méthode de Newton}

Soit $F$ un corps muni d'une valeur absolue ultramétrique
et non triviale pour laquelle il est complet. 
On s'intéresse dans ce chapitre aux équations polynomiales
à une ou plusieurs variables à coefficients dans~$F$.

\subsection{}
Isaac Newton, 
dans sa \emph{Méthode des fluxions} (1671, mais seulement publiée en 1736)
puis dans une lettre du 13 juin 1676 à son ami Henry Oldenburg,
diplomate et homme de sciences, 
a introduit la méthode 
pour résoudre les équations polynomiales
qui porte aujourd'hui son nom. Le contexte est la résolution
d'équations dépendant d'un paramètre --- ce qu'il appelle « équations affectées ». Citons cette lettre \citep{Turnbull-1960}:
\begin{quote}
Les extractions de racines affectées, d'équations à plusieurs termes littéraux, ressemblent par la forme à leurs extractions en nombres, 
mais la méthode de Viète et de notre compatriote Oughtred 
est moins adaptée à cet usage. 
C'est pourquoi j'ai été amené à en concevoir une autre…
\end{quote}

Partant d'une solution « approchée » $a$ d'une équation $f(x)=0$,
Newton récrit l'équation  sous la forme $f(a+y)=0$ en l'inconnue~$y$
et la développe sous la forme $g(y)=0$;
si $g(y)=f(a)+f_1(a) y+f_2(a) y^2+\dots $,
il considère que $b=-f(a)/f_1(a)$ est une solution approchée de cette
dernière équation, puis il itère le procédé.

La présentation moderne de cette méthode, telle qu'elle a pris forme aux 18\textsuperscript e et 19\textsuperscript e siècle (avec notamment Lagrange
et Fourier), consiste à 
itérer l'application 
\[ a \mapsto a- \frac{f(a)}{f'(a)}. \]
Par rapport au texte de Newton, elle combine deux modifications importantes,
tirées des travaux de Joseph Raphson et Thomas Simpson.  
Tout d'abord, la suite $(a_n)$ ainsi construite
est vue comme une suite d'approximations d'une solution de l'équation initiale,
alors que Newton considérait plutôt la suite $(a_1,a_2-a_1,\dots)$,
chaque terme étant vu comme solution approchée d'une équation différente.
D'autre part, le dénominateur~$f'(a)$ est interprété comme une différentielle
(une « fluxion » !) et non comme le terme~$f_1(a)$ 
dans le développement $f(a+x)=f(a)+f_1(a)x+\dots$.
Je renvoie aux deux articles de \citet{Kollerstrom-1992} 
et \citet{Christensen-1996}
pour plus de détails sur cette histoire.

Il ne s'agira pas ici de développer le succès de cette méthode
en analyse numérique, en théorie des systèmes dynamiques 
(de la méthode KAM aux plus médiatiques ensembles de Julia et Mandelbrot),
mais de montrer son impact dans le contexte des corps
munis d'une valeur absolue ultramétrique.

On commence par expliquer la méthode de Newton en dimension~1.

\begin{prop}
Soit $\phi \in F[T]$ un polynôme dont les coefficients
appartiennent à l'anneau de valuation~$A$ de~$F$.
Soit $a\in A$ tel que 
$\abs{\phi(a)} < \abs{\phi'(a)}^2$;
la méthode de Newton pour~$\phi$ issue du point~$a$
converge vers un élément $\alpha\in A$
qui est l'unique élément de~$A$
tel que $\phi(\alpha)=0$ et $\abs{\alpha-a}< \abs{\phi'(a)}$;
plus précisément, on a $\abs{\alpha-a}=\abs{\phi(a)/\phi'(a)}$.
De plus, la convergence est \emph{quadratique:} 
on a 
$\abs{a_n-\alpha} \leq \abs{\phi'(a)}\cdot (\abs{\phi(a)/\phi'(a)^2})^{2^n}$ 
pour tout entier~$n\geq0$.
\end{prop}
\begin{proof}
Posons $c=\abs{\phi(a)/\phi'(a)}$; les hypothèses
entraînent donc $c< \abs{\phi'(a)}\leq 1$. 
On écrit alors la formule de Taylor à l'ordre~$1$,
\[ \phi(a+T)=\phi(a)+\phi'(a) T +  T^2 \psi(T), \]
où $\psi \in A[T]$. 
Soit $b=-\phi(a)/\phi'(a)$; on a $\abs b =c < 1$ par hypothèse,
d'où $\abs{\phi(a+b)} = c^2 \abs{\psi(b)}$.
Puisque $\psi$ est à coefficients dans~$A$, 
on a donc $\abs{\psi(b)}\leq 1$,
d'où $\abs{\phi(a+b)} \leq c^2 $.
Par ailleurs, en écrivant 
\[ \phi'(a+T) = \phi'(a) + T \lambda (T), \]
où $\lambda \in A[T]$, 
on constate que 
\[ \abs{\phi'(a+b)-\phi'(a)}\leq \abs{b} = c < \abs{\phi'(a)}, \]
si bien que 
$\abs{\phi'(a+b)}=\abs{\phi'(a)}$.  
Posons $a_1=a+b$. On a donc 
$a_1\in A$,  $\abs{a_1-a}=c$, 
$\abs{\phi'(a_1)}=\abs{\phi'(a)}$, 
et $\abs{\phi(a_1)}\leq c^2$.
En particulier, on a 
\[ \abs{\phi(a_1)/\phi'(a_1)} \leq c^2 / \abs{\phi'(a_1)}
 = c^2 / \abs{\phi'(a)} < c < \abs{\phi'(a)} = \abs{\phi'(a_1)} \]
et 
\[ \abs{\phi(a_1)/\phi'(a_1)^2} \leq c^2 / \abs{\phi'(a)}^2 
	= \abs{\phi(a)/\phi'(a)^2} ^2 . \]

Cela entraîne qu'il existe une unique suite~$(a_n)$
dans~$A$ telle que $a_{0}=a$, 
$\abs{\phi(a_n)}<\abs{\phi'(a_n)}^2$
et $a_{n+1}=a_n-\phi(a_n)/\phi'(a_n)$
pour tout~$n$.
Plus précisément, on a $\abs{\phi'(a_n)}=\abs{\phi'(a)}$,
\[ \abs{a_{n+1}-a_n} \leq \abs{\phi(a_n)}/ \abs{\phi'(a_n)}\leq c , \]
et 
\[ \abs{\phi(a_{n+1})} \leq \abs{\phi(a_n)/\phi'(a)}^2  \]
pour tout~$n\in\N$.

Les inégalités
\[ \abs{\phi(a_{n+1})/\phi'(a_{n+1})^2} \leq (\abs{\phi(a_n)/\phi'(a_n)^2})^2\]
entraînent la majoration
\[ \abs{\phi(a_n)/\phi'(a_n)^2} \leq \abs{\phi(a)/\phi'(a)^2}^{2^n}, \]
pour tout $n\in\N$.
Posons $h=\abs{\phi(a)/\phi'(a)^2}$, de sorte que 
$\abs{a_{n+1}-a_n} \leq \abs{\phi'(a)} h^{2^n}$ pour tout entier~$n$.
En particulier, la suite $(a_{n+1}-a_n)$ tend vers~$0$;
comme $F$ est complet, la suite~$(a_n)$ converge;
sa limite est un élément~$\alpha$ de~$A$ tel
que $\abs{\phi(\alpha)}=0$ et $\abs{\alpha-a} \leq c$. 
Comme $\abs{\phi'(a_n)} =\abs{\phi(a)}$ pour tout~$n$,
on a aussi $\abs{\phi'(\alpha)}=\abs{\phi(a)}$.
Enfin, 
\[ \abs{\alpha-a_n} \leq \sup_{m\geq n} \abs{a_{m+1}-a_m} \leq \abs{\phi'(a)}
 h^{2^n}. \]

Soit $b$ un élément de~$A$ tel que $\abs {b-a}<\abs{\phi'(a)}$ et $\phi(b)=0$.
En écrivant $\phi(b)=\phi(a)+(b-a) \phi'(a)+(b-a)^2 \psi(b-a)$,
il vient $(b-a) (\phi'(a)+(b-a)\psi(b-a))=-\phi(a)$.
On a $\abs{\psi(b-a)}\leq 1$, car $\psi\in A[T]$ et $b-a\in A$;
comme $\abs {b-a}< \abs{\phi'(a)}$, 
on a donc $\abs{\phi'(a)+(b-a)\psi(b-a)}=\phi'(a)$,
d'où l'égalité $\abs {b-a}=\abs{\phi(a)/\phi'(a)}$.

En particulier, $\abs{\alpha-a}=\abs{\phi'(a)}$.

En développant la relation $\phi(\alpha+T)=\phi'(\alpha) T + T^2\theta(T)$,
où $\theta\in A[T]$, on obtient l'égalité
$(b-\alpha)(\phi'(\alpha) + (b-\alpha) \theta(b-\alpha))=0$.
Puisque 
$\abs{b-\alpha} =\abs{(b-a)-(b-a)}<\abs{\phi'(a)}$ et $\abs{\theta(b-\alpha)}\leq 1$, on a $\phi'(\alpha)+(b-\alpha)\theta(b-\alpha)\neq0$.
Par suite, $b-\alpha=0$ et $b=\alpha$.
\end{proof}

\begin{coro}
Soit $\phi\in A[T]$ et soit $a\in A$ tel que $\abs{\phi(a)}<1$
et $\abs{\phi'(a)}=1$. Si $F$ est complet, 
il existe un unique élément
$\alpha\in A$ tel que $\abs{a-\alpha}<1$ et $\phi(\alpha)=0$.
On a de plus $\abs{a-\alpha}=\abs{\phi(\alpha)}$.
\end{coro}
Observons que l'hypothèse signifie que lorsqu'on considère
le polynôme $f\in k[T]$ déduit de~$\phi$ par réduction modulo
l'idéal maximal de~$A$,
l'image $\overline a$ de~$a$ est racine de~$f$ mais pas de~$f'$.

\begin{exem}\label{exem.teichmuller}
Prenons par exemple un polynôme de la forme $\phi=T^d-a$,
de sorte que $\phi'=dT^{d-1}$.
Supposons $a \notin M$, de sorte que sa réduction $\overline a$
n'est pas nulle. Soit $c\in k$ tel que $c^d=\overline a$.
D'après la proposition, une condition suffisante 
pour qu'il existe un élément $\alpha\in A$ tel que
tel que $\alpha^d = a$ et $\overline\alpha=c$
est qu'il existe $\alpha\in A$ tel que $\overline\alpha=c$
et $\abs{\alpha^d-a}<\abs d^2$.

\emph a)
Si $d\cdot 1_A\not\in M$, c'est-à-dire si $d$ n'est pas multiple
de la caractéristique de~$k$, on a $\abs d=1$,
donc tout élément $\alpha\in A$ tel que $\overline\alpha=c$
vérifie le critère.

Pour un exemple concret, on peut prendre $\phi = T^2+1$
dans le corps $\Q_p$,  où $p$ est un nombre premier 
tel que $p\equiv 1\pmod 4$: le polynôme~$\phi$
est alors scindé dans~$\Q_p$. 

Observons que le polynôme $T^{p-1}-1$ est scindé sur~$\F_p$. 
Il existe ainsi pour tout élément $m\in\F_p^\times$
un unique élément $\theta_m\in\Z_p$ tel que $\overline{\theta_m}=m$
et $\theta_m^{p-1}=1$.  Ces éléments s'appellent parfois
les représentants de Teichmüller de~$\F_p^\times$.
Ils forment un sous-groupe cyclique d'ordre~$p-1$.

\emph b)
Si $d\cdot 1_A\in M$, il peut ne pas exister d'élément~$\alpha\in A$
tel que $\alpha^d=a$ et $\overline\alpha=c$.
C'est par exemple le cas pour l'équation $T^2+1$ sur~$\Q_2$,
avec $c=1$. En effet, si $\alpha^2+1=0$, on a
$(\alpha-1)^2=\alpha^2-2\alpha+1=-2\alpha$;
alors $v_2(\alpha)=0$ et $2 v_2(\alpha-1)=1$,
ce qui est absurde. 

En revanche, si $\abs{a-1}<\abs d^2$, on peut prendre $\alpha=1$
dans le critère, de sorte que $a$ est une puissance $d$-ième.
\end{exem}

\subsection{}
La méthode de Newton fonctionne également à plusieurs variables.
Munissons $F^n$ de la norme donnée par $\norm{(a_1,\dots,a_n)}=\sup(\abs{a_1},\dots,\abs{a_n})$. 
Soit $\phi_1,\dots,\phi_n\in F[T]$ des polynômes à coefficients dans~$F$
en $n$~indéterminées $T=(T_1,\dots,T_n)$ et soit 
$\Phi=(\phi_1,\dots,\phi_n)\colon F^n\to F^n$ l'application polynomiale
qu'on en déduit. 
Sa différentielle $D\Phi$ est une application
de~$F^n$ dans l'espace des matrices de type~$n\times n$ 
dont les  coefficients sont donnés par les dérivées partielles formelles
$\partial\phi_i/\partial T_j$ des  polynômes~$\phi_i$.
On définit également le jacobien de~$\Phi$
comme l'application polynomiale $J\Phi \colon a \mapsto \det(D\Phi(a))$.

La méthode de Newton pour~$\Phi$ est l'itération de l'application
rationnelle de~$F^n$ dans lui-même donnée par 
\[    a \mapsto \Phi(a) - (D\Phi(a))^{-1} \Phi(a), \]
définie pourvu que $J\Phi(a)\neq0$.

\begin{theo}
Soit $F$ un corps muni d'une valeur absolue ultramétrique
pour laquelle il est complet. 

Soit $\Phi\colon F^n\to F^n$ une application polynomiale
définie par des polynômes à coefficients
dans l'anneau de valuation~$A$ de~$F$;

Soit $a\in A^n$ tel que $\norm{\Phi(a)} <\abs{J\Phi(a)}^2$,
il existe un unique point $\alpha\in F^n$ tel que $\norm{\alpha-a}\leq \norm{\Phi(a)}/\abs{J\Phi(a)}$ et $\Phi(\alpha)=0$.
Il est obtenu par la méthode de Newton issue de~$a$
et vérifie $\norm{a-\alpha}\leq \norm{\Phi(a)} / \abs{J\Phi(a)}$.
\end{theo}
\begin{proof}
La preuve est essentiellement la même qu'en dimension~1;
la formule de Taylor est le seul point qui demande quelques arguments
supplémentaires.
A priori, on peut écrire 
\[ \Phi(a+T) = \Phi(a) + D\Phi(a)\cdot T + \Psi(T), \]
où $\Psi$ est donné par des polynômes à coefficients dans~$A$
et dont chaque monôme est de degré au moins~2.
Par suite, on a  $\norm {\Psi(b)}\leq \norm b^2$ pour tout $b\in F^n$.
De même, la formule de la comatrice entraîne
que $J\Phi(T) \cdot D\Phi(T)^{-1}$ est une matrice  dont les entrées
sont des polynômes à coefficients dans~$A$;
en particulier, 
\[  \norm{D\Phi(a)^{-1}  \cdot b} \leq \abs{J\Phi(a)}^{-1} \norm b \]
pour tout $b\in F^n$.
Le reste de la démonstration est laissé aux lecteurs et lectrices.
\end{proof}

\subsection{}
Avant d'énoncer le corollaire suivant, 
faisons quelques rappels sur le résultant de deux polynômes.
Soit $P, Q\in A[T]$ deux polynômes à coefficients dans un anneau commutatif~$A$
et soit $p,q$ deux entiers naturels tels que $\deg(P)\leq p$ et $\deg(Q)\leq q$.
Le résultant $\Res_{p,q}(P,Q)$ est alors défini comme le déterminant
d'une matrice (« de Sylvester ») de taille $p+q$ consistant
à écrire, sur $q$ lignes, les coefficients de~$P$, décalés
d'une colonne d'une ligne sur la suivante,
puis, sur~$p$ lignes, les coefficients de~$Q$, 
décalés d'une colonne d'une ligne sur la suivante.
De manière plus conceptuelle, et surtout plus éclairante, c'est
le déterminant de l'application linéaire 
\[ A[T]_{<q} \times A[T]_{<p} \to A[T]_{<p+q} , \quad (U,V) \mapsto UP+VQ , \]
où l'espace vectoriel source est muni de la base
$(1,T,\dots,T^{q-1},1,\dots,T^{p-1})$ et l'espace vectoriel but
de la base $(1,T,\dots,T^{p+q-1})$.

Supposons que l'anneau~$A$ soit un corps~$F$.
Alors, le résultant s'annule si et seulement si cette application linéaire  
n'est pas injective, et l'on vérifie que cela se produit si et seulement
si l'une des deux conditions suivantes est vérifiée :
\begin{enumerate}
\item Les polynômes~$P$ et~$Q$ ne sont pas premiers entre eux;
\item On a $\deg(P)<p$ et $\deg(Q)<q$.
\end{enumerate}
Posons en effet $D=\gcd(P,Q)$ et écrivons 
$P=DP_1$ et $Q=DQ_1$; on a la relation $Q_1P-P_1Q=0$
qui prouve que l'application linéaire précédente n'est pas injective
lorsque $\deg(Q_1)<q$ et $\deg(P_1)<q$. Cela se produit
si $\deg(D)>0$ ou si l'on a simultanément $\deg(Q)<q$ et $\deg(Q)<p$.
Inversement, supposons que les polynômes~$P,Q$ sont premiers entre eux
et que l'application linéaire considérée n'est pas injective;
considérons un élément non nul $(U,V)$ de son noyau; la relation $UP+VQ=0$
et la coprimalité de~$P$ et~$Q$ entraîne que $Q$ divise~$U$
et que~$P$ divise~$V$. On en déduit $\deg(Q)<q$ et $\deg(P)<p$.

Par construction, le résultant $\Res_{p,q}(P,Q)$
est une application polynomiale 
en les coefficients de~$P$ et~$Q$; en particulier,
si $\phi\colon A\to B$ est un morphisme d'anneaux commutatifs,
et si $P^\phi$ et~$Q^\phi$ désignent les polynômes de~$B[T]$
obtenus en appliquant~$\phi$ aux coefficients de~$P$ et~$Q$,
on a \[ \Res_{p,q}(P^\phi,Q^\phi)=\phi(\Res_{p,q}(P,Q)). \]
Observons que l'on peut fort bien avoir $\deg(P^\phi)<\deg(P)$
ou $\deg(Q^\phi)<\deg(Q)$, si bien que la généralité
consistant à ne pas fixer~$p=\deg(P)$ et~$q=\deg(Q)$ est ici nécessaire.

\begin{coro}
Soit $\phi \in A[T]$ un polynôme de degré~$n$ et soit 
$\psi,\eta\in A[T]$ deux polynômes 
tels que 
$\deg(\phi)=\deg(\psi)+\deg(\eta)$
et $\norm{\phi -\psi\eta}<\abs{\Res(\psi,\eta)}^2$,
où $\Res(\psi,\eta)$ désigne le résultant des polynômes~$\phi$ et~$\psi$
(relativement à leurs degrés).
Il existe alors un unique couple $(\gamma^*,\eta^*)$
dans~$A[T]$
tels que $\phi=\gamma^*\eta^*$, $\deg(\psi^*-\psi)<\deg(\psi)$,
$\norm{\psi^*-\psi}<\abs{\Res(\psi,\eta)}$,
$\deg(\eta^*-\eta)<\deg(\eta)$ et $\norm{\eta^*-\eta}<\abs{\Res(\psi,\eta)}$.
\end{coro}

\begin{proof}
Posons $p=\deg(\psi)$, $q=\deg(\eta)$ et $n=p+q=\deg(\phi)$.
On identifie l'espace affine~$U$ des polynômes~$g$
tels que $\deg(g-\psi)<\deg(\psi)$ à l'espace vectoriel
à l'espace vectoriel~$F^p$,
celui~$V$ des polynômes~$h$ tels que $\deg(h-\eta)<\deg(\eta)$
à l'espace vectoriel~$F^q$,
et celui~$W$ des polynômes~$f$ tels que $\deg(f-\phi)<\deg(\phi)$
à l'espace vectoriel~$F^n$. 
Considérons l'application~$\Phi$ de $U\times V$ dans~$W$
donnée par $(g,h)\mapsto gh$.
Elle est polynomiale.
Sa différentielle en un couple $(g,h)$ est l'application
linéaire $(u,v)\mapsto uh+v\eta$ de $F[T]_{<q}\times F[T]_{<p}$
dans~$F[T]_{<n}$; par suite, son jacobien $J\Phi(g,h)$
en un couple~$(g,h)$ est égal au 
\emph{résultant} $\Res(g,h)$.
L'assertion résulte donc du théorème.
\end{proof}

\begin{coro}
Soit $\phi\in F[T]$ un polynôme irréductible et unitaire tel que $\phi(0)\in A$.
Alors $\phi\in A[T]$.
\end{coro}
\begin{proof}
Notons $\phi=\sum_{m=0}^n a_m T^m$. Il s'agit
de prouver que $\abs{a_m}\leq 1$ pour tout entier~$m$;
notons $p$ le plus grand entier 
tel que $\abs{a_p}=\sup(\abs{a_0},\dots,\abs{a_n})$.
Si $p=0$ ou $p=n$, on a $\abs{a_p}\leq 1$ par hypothèse,
d'où $\sup(\abs{a_0},\dots,\abs{a_n})\leq 1$.
Raisonnons par l'absurde; on a donc $0<p<n$.
On écrit 
\[ a_p^{-1}\phi = b_n T^n+\dots + b_{p+1} T^{p+1}+T^p
 + b_{p-1}T^{p-1}+\dots+b_0, \]
de sorte que $\abs{b_0},\dots,\abs{b_{p-1}}\leq 1$
et $\abs{b_{p+1}},\dots,\abs{b_n} <1$.
Posons alors $\psi=b_nT^{n-p}+1$
et $\eta=T^p+b_{p-1}T^{p-1}+\dots+b_0$.
Ce sont des éléments de~$A[T]$; le polynôme 
$ a_p^{-1}\phi - \psi\eta$ est de degré~$<n$ et ses coefficients
sont tous de valeur absolue~$<1$.
De plus, le résultant $\Res(\psi,\eta)$ (relativement à leurs degrés~$n-p$
et~$p$)
est un élément de~$A$
dont l'image dans le corps résiduel~$k$ de~$F$
est le résultant des polynômes $\overline\psi=1+\overline{b_n}T^{n-p}$
et $\overline\eta= T^p$ relativement aux degrés~$n-p$ et~$p$.
Comme ces polynômes n'ont pas de racine commune et comme
$p=\deg(\overline\eta)$,
ce résultant n'est pas nul.  Par suite, $\abs{\Res(\psi,\eta)}=1$.
Le théorème entraîne donc qu'il existe un unique couple $(\psi^*,\eta^*)$
de polynômes tels que $\deg(\psi^*-\psi)<n-p$,
$\deg(\eta^*-\eta)<p$, $\norm{\psi^*-\psi}<1$, $\norm{\eta^*-\eta}<1$
et $a_p^{-1}\phi = \psi^*\eta^*$.
Alors, $\phi = a_p \psi^*\eta^*$, ce qui contredit l'irréductibilité de~$\phi$.
\end{proof}

Donnons tout de suite une application importante.
\begin{theo}
Soit $F$ un corps muni d'une valeur absolue ultramétrique
pour laquelle il est complet
et soit $E$ une extension algébrique de~$F$.
Il existe une unique valeur absolue sur~$E$ qui prolonge celle de~$F$;
elle applique un élément~$a\in E$ de polynôme minimal~$P$
sur $\abs{P(0)}^{1/\deg(P)}$.
\end{theo}
\begin{proof}
Traitons d'abord le cas où la valeur absolue de~$F$ est triviale;
démontrons que la valeur absolue triviale est la seule valeur
absolue sur~$F$ qui prolonge celle de~$E$
Soit $b\in E\setminus F$ 
et soit $a_0,\dots,a_n$ des éléments de~$F$, non tous nuls,
tels que $a_n b^n+\dots+a_0=0$; il s'agit de prouver que $\abs b=1$.
Comme la valeur absolue de~$E$ est nécessairement ultramétrique,
il existe deux entiers~$p<q$ tels que $\abs{a_pb^p}=\abs{a_q b^q}$
et $a_p,a_q\neq 0$.
Alors $\abs{a_p}=\abs{a_q}=1$, puis $\abs{b}^p=\abs{b}^q$,
d'où $\abs b=1$.

Supposons maintenant que la valeur absolue de~$F$ n'est pas triviale.
Comme $E$ est réunion d'extensions finies, il suffit de traiter
le cas où $E$ est une extension finie de~$F$.
Alors, $E$ est un $F$-espace vectoriel de dimension finie.

Soit $m_1$ et~$m_2$ des valeurs absolues sur~$E$ qui prolongent celle de~$F$;
alors $m_1$ et $m_2$ sont des normes sur~$E$.
Comme la dimension de~$E$ est finie,
comme toutes les normes sur~$E$ sont équivalentes et
les valeurs absolues~$m_1$ et~$m_2$ définissent la même topologie sur~$E$. 
Il existe donc un nombre réel~$\rho>0$
tel que $m_1(a)=m_2(b)^\rho$ pour tout $b\in E$.
En considérant un élément $b\in F$ tel que $\abs b\neq 1$,
on obtient $\rho=1$. Cela démontre 
que la valeur absolue de~$F$ se prolonge d'au plus une façon en une
valeur absolue de~$E$.

Pour $a\in E$, considérons $a$ comme un endomorphisme $F$-linéaire de~$E$
et posons $m(a)=\abs{\det(a)}^{1/d}$, où $d=[E:F]$.
On a $m(0)=0$; $m(1)=1$; $m(ab)=m(a)m(b)$ pour $a,b\in E$;
pour $a\in F$, on a $m(a)=\abs{\det(a)}^{1/d}=\abs{a^d}^{1/d}=\abs a$.
Il reste à démontrer que $m(a+b)\leq \sup(m(a),m(b))$ pour tout $a,b\in E$;
pour cela, il suffit de prouver que $m(1+a)\leq 1$ si $a\in E$
et $m(a)\leq 1$.
Soit $\phi\in F[T]$ le polynôme minimal de~$a$; si son degré est~$e$,
le corps~$F$ est de façon naturelle un espace vectoriel de dimension~$d/e$
sur son sous-corps~$F(a)$. En considérant une base de~$F$ sur~$F(a)$,
on voit que le polynôme caractéristique de~$a$  est égal à $\phi^{d/e}$.
Par suite, $\det(a)=\phi^{d/e}(0)$ et  $m(a)=\abs{\phi(0)}^{1/e}$.
Comme le polynôme minimal de~$1+a$ est égal à~$\phi(T-1)$,
on a $m(1+a)=\abs{\phi(-1)}^{1/e}$.
Par hypothèse, $\phi$ est irréductible, unitaire, et $\phi(0)\in A$
puisque $\abs{\phi(0)}=m(a)^e \leq 1$.
D'après le corollaire, $\phi \in A[T]$; alors $\phi(-1)\in A$
et $m(1+a)\leq 1$.

La dernière relation a été établié au cours de la démonstration.
\end{proof}

\subsection{}
Soit $F$ un corps muni d'une valeur absolue~$\abs{\,\cdot\,}$
ultramétrique pour laquelle il est complet.
Supposons que $F$ n'est pas algébriquement clos,
considérons-en une clôture algébrique~$\overline F$ 
et munissons-la de l'unique valeur absolue qui prolonge celle de~$E$. 

Supposons d'abord que $[\overline F:F]$ est fini;
dans ce cas, $\overline F$ est complet. 
Notons que cette situation est plutôt exceptionnelle.
En effet, un théorème d'\cite{ArtinSchreier-1927} affirme que 
cela ne se produit que si $F$ est « réel clos »
($F$ est ordonnable, et les polynômes
vérifient le théorème des valeurs intermédiaires)
auquel cas $[\overline F:F]=2$ et $\overline F=F(\sqrt{-1})$.
En particulier, $F$ est de caractéristique zéro
et les éléments négatifs n'ont pas de racine carrée.
Vérifions aussi que la valeur absolue de~$F$ est triviale.
Comme elle est ultramétrique, le théorème d'Ostrowski entraînerait
sinon que $F$ contient un corps $p$-adique. La contradiction
vient de ce que d'après l'exemple~\ref{exem.teichmuller},
tout entier 
congru à~$1$ modulo~$p$ (modulo~$8$ si $p=2$)
possède une racine carrée dans~$\Q_p$,
donc dans~$F$, même s'il est strictement négatif.

Ainsi, on peut se concentrer sur le cas où~$[\overline F:F]$ est infini.
Dans ce cas, $\overline F$ n'est pas complet.
(La méthode, un peu technique, consiste à construire une série
rapidement convergente 
dont les termes successifs sont de degrés de plus en plus grand,
de sorte que la limite n'est pas algébrique.)
On peut alors
considérer son complété $\widehat{\overline F}$,
et la proposition suivante affirme que c'est un corps algébriquement clos.

\begin{prop}
Si $F$ est un corps algébriquement clos
muni d'une valeur absolue,  alors son complété
$\widehat F$ est encore algébriquement clos.
\end{prop}
\begin{proof}
Compte tenu du \S\ref{ss.archimedien}, on se borne
à traiter le cas où la valeur absolue de~$F$ est ultramétrique.
Considérons alors  un polynôme unitaire irréductible
$\phi\in \widehat{F}[T]$; démontrons que son degré~$d$
est égal à~$1$. Pour cela, il suffit de démontrer qu'il possède
une racine, racine que nous allons construire à l'aide de la méthode
de Newton en partant d'une racine d'un polynôme à coefficients dans~$F$
qui est assez proche de~$\phi$.

Nous allons nous contenter de traiter le cas où $\phi$ est
un polynôme \emph{séparable}, c'est-à-dire que 
ses racines de~$\phi$ dans une clôture algébrique de~$\widehat F$ sont simples;
de manière équivalente, les polynômes~$\phi$ et~$\phi'$ sont premiers entre eux.
Cette condition est automatique lorsque $F$ est de caractéristique zéro;
elle ne l'est pas lorsque $F$ est de caractéristique~$p$,
le polynôme~$\phi$ est alors un polynôme en~$T^p$ et
on élimine ce cas sans trop de peine
par des techniques qui n'apportent rien au propos de ce texte.
Le fait que $\phi$ soit séparable se traduit aussi par le fait
que le \emph{discriminant} de~$\phi$,
défini comme le résultant $\Res_{d,d-1}(\phi,\phi')$, n'est pas nul.

Quitte à remplacer~$\phi$ par un polynôme de la forme $\phi(cT)/c^d$,
on peut aussi supposer que tous ses coefficients sont de valeur
absolue au plus~$1$.

Soit $\psi\in F[T]$ un polynôme unitaire de degré~$d$
que nous prendrons assez proche de~$\phi$.
Ses coefficients  sont en particulier de valeur absolue au plus~$1$.
Notons $a_1,\dots,a_d$ les racines de~$\psi$ dans~$F$ ; elles
vérifient $\abs {a_1},\dots,\abs{a_d} \leq 1$.
(Si $\psi=\sum_{n=0}^d c_n T^n$, on a $\abs{c_n a^n}\leq\abs a^n$
pour tout~$n$, et $\abs{c_d a^d}=\abs a^d$ puisque $c_d=1$;
si $a\in F$ vérifie $\abs a >1$, le terme d'indice~$d$ est de plus grande
valeur absolue, à savoir~$\abs a^d$, et c'est le seul, 
si bien que $\abs{\psi(a)}=\abs a^d$; en particulier, $\psi(a)\neq0$.)

 
Soit $a$ l'un quelconque des~$a_j$.
Sous réserve que l'on ait
$\abs{\phi(a)} < \abs{\phi'(a)}^2$,
la méthode de Newton fournit
un élément~$\alpha\in \widehat F$ tel que $\abs{\alpha-a}<1$
et $\abs{\phi(\alpha)}=0$.  Il suffit donc de justifier
que cette inégalité est vérifiée lorsque que $\norm{\phi-\psi}$
est assez petit.

Comme $\abs a\leq 1$ et $\psi(a)=0$, 
on a $\abs{\phi(a)}=\abs{\phi(a)-\psi(a)}\leq \norm{\phi-\psi}$.
Démontrons comment minorer $\abs{\phi'(a)}$.

En écrivant
\[ \abs{\phi'(a)-\psi'(a)}\leq \norm{\phi'-\psi'}\leq\norm{\phi-\psi}, \]
on voit que
$ \abs{\phi'(a)} = \abs{\psi'(a)} $ si $\norm{\phi-\psi}<\abs{\psi'(a)}$;
il suffit donc de minorer cette quantité.
Pour cela, on fait appel aux formules classiques calculant le
le discriminant de~$\psi$ en fonction de ses racines :
\[ D(\psi)=\prod_{1\leq i<j\leq d} (a_j-a_i)^2
 = (-1)^{(d-1)(d-2)/2} \prod_{j=1}^d \psi'(a_j). \]
Comme $\abs{a_j}\leq 1$ pour tout~$j$ et $\norm\psi\leq 1$,
on a $\abs{\psi'(a_j)}\leq 1$ pour tout~$j$, si bien que 
$\abs{D(\psi)}=\prod_{j=1}^d \abs{\psi'(a_j)}\leq \abs{\psi'(a)}$.
Enfin, la définition du discriminant comme déterminant d'une matrice
formée des coefficients de~$\psi$ et~$\psi'$ l'exprime
comme un polynôme à coefficients entiers en les coefficients de~$\psi$.
Grâce à l'inégalité ultramétrique,
on constate donc que $\abs{D(\phi)-D(\psi)}\leq \norm{\phi-\psi}$.
Si de plus $\norm{\phi-\psi}<\abs{D(\phi)}$, il en résulte
que $\abs{D(\psi)}=\abs{D(\phi)}$.

En conclusion, dès que~$\psi$ est choisi de sorte que
$\norm{\phi-\psi}<\abs{D(\phi)}$, on a
l'inégalité $\abs{\phi'(a)}\geq \abs{D(\phi)}$.

Comme on a aussi $\abs{D(\phi)}\leq 1$,
il suffit de prendre $\psi$ tel que
$\norm{\phi-\psi}<\abs{D(\phi)}^2$ pour que la méthode
de Newton issue de n'importe laquelle de ses racines
fournisse une racine de~$\phi$.
\end{proof}

\section{Polytopes et polygones de Newton}
\subsection{}
Soit $\phi = \sum_{m\in\N^n} a_m T^m \in F[T]$ un polynôme 
en $n$~indéterminées $T_1,\dots,T_n$ et à coefficients dans~$F$.
Pour tout $r\in\R_+^n$ et $m\in\N^n$,  on pose  $r^m=r_1^{m_1}\dots r_n^{m_n}$
et on définit 
\[ \norm \phi _r = \sup_{m\in\N^n} \abs{a_m} r^m. \]
La notation à l'aide de multi-indices a pour but de laisser
le lecteur ou la lectrice que ça arrangerait croire qu'il n'y a
qu'une seule indéterminée.

On appelle support de~$\phi$ l'ensemble (fini) $S_\phi$ des $m\in\N^n$
tels que $a_m\neq 0$. 
On appelle \emph{polytope de Newton} de~$\phi$
l'enveloppe convexe $\Pi_\phi$  de~$S_\phi$.
On dit qu'un point $x\in\R^n$ est un sommet de~$\Pi_\phi$ 
s'il existe un hyperplan affine~$H$ de~$\R^n$
tel que $H\cap\Pi_\phi=\{x\}$; si $f$ est une forme affine définissant~$H$,
la restriction de~$f$ à~$\Pi_\phi$ a un signe constant.
On démontre que les sommets de~$\phi$ sont des points de~$S_\phi$
et que $\Pi_\phi$ est l'enveloppe convexe de ses sommets.

Les égalités $S_\phi=\emptyset$,  $\Pi_\phi=\emptyset$ et $\phi=0$
sont équivalentes.

\begin{prop}
L'application $\phi\mapsto \norm\phi_r$  est une valeur
absolue sur l'anneau $F[T]$ telle que $\norm{a\phi}_r=\abs a \norm\phi_r$
pour tout $a\in F$ et tout  $\phi\in F[T]$.
\end{prop}
Le cas où $n=1$ et $r=1$ est l'exemple~\ref{exem.gauss}
de la valeur absolue de Gauss;
la démonstration ci-dessous est plus géométrique.

\begin{proof}
Les égalités $\norm0_r=0$ et $\norm1_r=1$ sont évidentes,
de même que l'égalité $\norm{a\phi}_r=\abs a\norm\phi_r$
pour $a\in F$ et $\phi\in F[T]$,
et l'inégalité ultramétrique $\norm{\phi+\psi}_r\leq \sup(\norm\phi_r
,\norm\psi_r)$ pour $\phi,\psi\in F[T]$.
Démontrons que $\norm\cdot_r$ est multiplicative.

Soit $\phi,\psi\in F[T]$; prouvons $\norm{\phi\psi}_r=\norm\phi_r\norm\psi_r$.
Il suffit de traiter le cas où $\phi,\psi$ sont non nuls.
Posons $\phi=\sum a_m T^m$, $\psi=\sum b_m T^m$ et $\phi\psi=\sum c_mT^m$.
Pour tout $m\in\N^n$, on a 
$ c_m = \sum_{p+q=m} a_p b_q $, de sorte que
\[ \abs{c_m}r^m \leq \sup_{p+q=m} \abs{a_p} \abs{b_q} r^m 
= \sup_{p+q=m} \abs{a_p}r ^p \, \abs{b_q} r^q 
\leq \norm \phi_r \norm\psi_r. \]
Inversement,
notons $S_{\phi,r}$ l'ensemble des $m\in\N^n$
tels que $\abs{a_m} r^m = \norm\phi_r$ et
définissons $S_{\psi,r}$ et $S_{\phi\psi,r}$ de façon analogue.
Soit $m\in\N^n$. S'il n'y a aucun couple  $(p,q)$
tel que $p\in S_{\phi,r}$, $q\in S_{\psi,r}$ et $p+q=m$, 
le calcul précédent montre que $\abs{c_m}r^m < \norm\phi_r\norm\psi_r$.
Par suite, $S_{\phi\psi,r}$ est contenu dans
la somme de Minkowski $S_{\phi,r}+S_{\psi,r}$.
Soit $m$ un sommet de l'enveloppe convexe
des $p+q$, pour $p\in S_{\phi,r}$ et $q\in S_{\psi,r}$;
alors il n'existe qu'un seul couple $(p,q)$ comme ci-dessus,
et le même calcul garantit 
$\abs{c_m}r^m=\abs{a_p}r^p \abs{b_q}r^q =\norm\phi_r \norm\psi_r$. 
Par suite, $\norm{\phi\psi}_r=\norm\phi_r\norm\psi_r$,
comme il fallait démontrer.
\end{proof}

\begin{rema}
Dans le cas particulier où $F$ est muni de la valeur absolue triviale,
la démonstration précédente met en évidence le résultat
géométrique suivant:
\emph{Soit $\phi,\psi$ des éléments non nuls de~$F[T]$.
Alors le polytope de Newton~$\Pi_{\phi\psi}$ de leur produit~$\phi\psi$ 
est la somme de Minkowski $\Pi_\phi+\Pi_\psi$ des polytopes de Newton~$\Pi_\phi$ et $\Pi_\psi$,
c'est-à-dire l'ensemble des $x+y$, pour $x\in\Pi_\phi$ et $y\in\Pi_\psi$.}

On peut en déduire des théorèmes d'irréductibilité
dans l'anneau $F[T]$: si, par exemple, le polytope de Newton~$\Pi_\phi$
n'est pas réduit à un point et n'est pas la somme de Minkowski $P+Q$ 
de deux polytopes non réduits à un point et à sommets entiers,
alors $\phi$ est irréductible dans~$F[T]$. \citet{Gao-2001}
montre le lien avec les critères classiques d'Eisenstein et de 
Dumas et donne quelques exemples concrets où le polytope de Newton
est contenu dans un triangle du plan.
\end{rema}

\subsection{}
Parce que nous avons un peu de mal à considérer que la loi
d'action de~$\R$ sur~$\R_+^*$ donnée par $(t,x)\mapsto x^t$ 
fait de~$\R_+^*$ un espace affine sur~$\R$,
on passe au logarithme 
et on associe à tout polynôme $\phi\in F[T]$ la fonction~$\tau_\phi$ sur~$\R^n$,
parfois appelée le polynôme tropical\footnote{%
La légende affirme que l'adjectif \emph{tropical} a été introduit
pour rendre hommage au mathématicien brésilien Imre Simon.
Il reflète \emph{de facto} une vision du monde eurocentrée.
Son utilisation est ainsi critiquée dans ce contexte mathématique, 
en particulier par des mathématicien·nes d'origine sud-américaine qui,
cependant, continuent à l'utiliser, peut-être  
dans l'attente d'une terminologie moins problématique; nous faisons de même.}
 associé à~$\phi$,
qui est définie par $\tau_\phi(x)=\log(\norm{\phi}_{e^x})$,
Explicitement si $\phi=\sum c_mT^m$, on a 
\[ \tau_\phi (x) =  \sup_{m\in S_\phi} (\log(\abs{c_m})+ m\cdot x) . \]
C'est une fonction convexe continue, affine par morceaux,
en tant que borne supérieure d'une famille finie de fonctions affines.
Si $\tau_\phi$ est affine sur un ouvert non vide~$V$ de~$\R^n$
il existe un unique point $m\in S_\phi$  tel que 
$\tau_\phi$ coïncide avec $x\mapsto \log(\abs {c_m})+m\cdot x$
sur~$V$. 
Sur l'ensemble des $t\in F^n$ tels que $\log(\abs t)\in V$, 
on a alors $\log(\abs{\phi(t)})  = \tau_\phi(\log(\abs t))$;
en particulier, $\phi$ ne s'annule pas sur cet ensemble.

\subsection{}
La transformée de Legendre de~$\tau_\phi$ est définie par
\[ \nu_\phi(p) = \sup_{x} \big( p \cdot x - \tau_\phi(x)\big), \]
où $x$ parcourt~$\R^n$.  Elle est convexe et semi-continue inférieurement, 
en tant que borne supérieure d'une famille de fonctions affines,
mais comme cette famille est infinie (ai-je supposé que l'espace
est de dimension~$>0$ ?…),
elle peut valoir~$+\infty$.
On l'appelle le \emph{polygone de Newton} de~$\phi$. 

Soit $m\in S_\phi$; 
on a donc $\tau_\phi(x)\geq m\cdot x +\log(\abs{c_m})$ pour tout $x\in\R^n$,
de sorte que $m\cdot x - \tau_\phi(x) \leq - \log( \abs{c_m})$,
si bien que $\nu_\phi(m)\leq -\log(\abs{c_m})$; en particulier,
$\nu_\phi(m)$ est finie. Il en résulte que $\nu_\phi$ est finie
et continue sur le polytope de Newton~$\Pi_\phi$ de~$\phi$, 
et  on peut démontrer qu'elle vaut~$+\infty$ en dehors.

En fait, $\nu_\phi$ est la plus grande
fonction $\nu_\phi$ sur~$\Pi_\phi$
qui soit convexe, semi-continue inférieurement,
et vérifie $\nu_\phi(m) \leq - \log (\abs{c_m})$ pour $m\in S_\phi$;
elle est affine par morceaux. 
Son épigraphe dans $\R^n\times \R$
est l'enveloppe convexe supérieure des points 
de la forme $(m,-\log(\abs{c_m}))$, pour $m\in S_\phi$.
En particulier, elle vaut $-\log(\abs{c_m})$ en tout sommet~$m$
de~$\Pi_\phi$.

\subsection{}
La dualité des fonctions convexes a une formulation
symétrique en termes de « sous-différentiels »:
\[ p \in \partial\tau_\phi (x) \quad\Leftrightarrow\quad
    x \in \partial \nu_\phi(p). \]
Comme les fonctions~$\tau_\phi$ et $\nu_\phi$ sont affines par morceaux,
cela fournit une dualité géométrique
entre domaines d'affinité de~$\tau_\phi$ et~$\nu_\phi$
pour lesquels la dimension est complémentaire.
On a vu le lien entre polytopes maximaux 
sur lesquels $\tau_\phi$ est affine et points anguleux de~$\nu_\phi$.
Inversement, il y a correspondance entre
points anguleux de~$\tau_\phi$ 
et polytopes maximaux sur lesquels~$\nu_\phi$ est affine.

\subsection{}
C'est dans une lettre d'Isaac Newton à Henry Oldenburg
qu'apparaît ce polygone, dans le contexte de la résolution
d'une équation en deux variables:
\[ y^6-5xy^5+(x^3/a)y^4-7a^2x^2y^2+6a^3x^3+b^2x^4 = 0 \]
pour laquelle Newton cherche à exprimer~$y$ en termes
d'une série de puissances de~$x$;
il fait à peu près le dessin ci-dessous, 
si ce n'est que nous plaçons les étiquettes et les points 
aux intersections et non dans les cases: les points
marqués sont les couples $(m,n)$ tels que $x^ny^m$ apparaît dans
le polynôme ci-dessus, et le  trait gras est une « règle » qu'on placerait
sous ces points, le plus haut possible.
\begin{figure}[htbp]
\centering
\begin{tikzpicture}
\draw [line width=1pt,fill = blue!20](6,0)--(4,3)--(0,4)--(0,3)--cycle;
\fill (6,0) circle (3pt);
\fill (5,1) circle (3pt);
\fill (4,3) circle (3pt);
\fill (2,2) circle (3pt);
\fill (0,3) circle (3pt);
\fill (0,4) circle (3pt);
\draw [line width=2pt] (0,3)--(6,0);
\foreach \y in {0,1,...,6}{
  \node (y=\y) at (\y,0) [below] {$y^\y$};}
\foreach \x in {0,1,...,4}{
  \node (x=\x) at (0,\x) [left] {$x^\x$};}
\draw (0,0) grid (6,4);
\end{tikzpicture}
\caption{L'exemple de Newton : 
$\phi =  y^6-5xy^5+(x^3/a)y^4-7a^2x^2y^2+6a^3x^3+b^2x^4 = 0 $}
\end{figure}
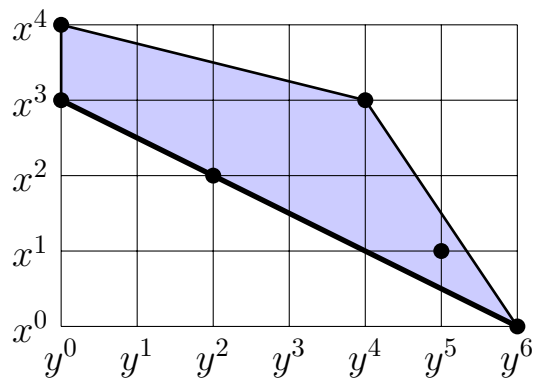
Newton extrait du polynôme les termes pour lesquels
le couple $(m,n)$ est sur le trait gras, ici :
\[ y^6-7a^2x^2+6a^3 x^3 =0 \]
qu'il réduit en $v^6-7v^2+6=0$, en posant $y=v\sqrt{ax}$.
C'est une équation « bicubique » 
dont les solutions sont les racines carrées de~$1, 2$ et $-3$;
Newton écrit alors
\[ y = +\sqrt{ax}, \quad -\sqrt {ax}, \quad +\sqrt{2ax}, \quad - \sqrt{2ax},\]
expressions qu'il voit comme le premier terme de la solution,
car il évite (sans même les mentionner) les solutions $y =\pm\sqrt{-3ax}$.
Pour chacun de ces termes~$p$, il récrit alors l'équation
pour $y-p$, et répète l'argument.

Dans notre langage, il convient de considérer l'anneau
des séries formelles $\C\lbra x\rbra$ à coefficients complexes,
ou plutôt son corps des fractions $F=\C\lpar x\rpar$, 
muni d'une valeur absolue ultramétrique pour laquelle 
$\abs a=1$ pour tout $a\in\C^*$ et $\abs x<1$; 
pour fixer les idées et simplifier les logarithmes ci-dessous,
nous prenons $\abs x=1/e$.
Le polynôme de deux variables devient alors le polynôme d'une
variable
\[  \phi =  T^6-5xT^5+(x^3/a)T^4-7a^2x^2T^2+(6a^3x^3+b^2x^4)  \]
à coefficients dans $F$.
Si on note $\phi=\sum c_m T^m$,
on a donc $\log(\abs {c_0})=-3$, $\log(\abs{c_2})=-2$,
$\log(\abs{c_4})=-3$ et $\log(\abs{c_5})=-1$.
Ainsi, à une symétrie par rapport à l'axe horizontal près,
le polygone de Newton de~$\phi$ est précisément
donné par le dessin que Newton a réalisé.

On devine dans l'exemple de Newton que
les racines~$y$ de~$\phi$,
y compris les deux que Newton n'a pas considérées,
auront pour valeur absolue $e^{-1/2}$.
C'est effectivement le cas: les pentes du polygone
de Newton sont les logarithmes des valeurs absolues des racines.

\begin{theo}\label{theo.newton-pol-pol}
Soit $F$ un corps muni d'une valeur absolue ultramétrique 
et soit $\phi=\sum c_j T^j\in F[T]$ un polynôme 
en une indéterminée tel que $c_0=\phi(0)\neq0$; soit $d$ son degré.

\begin{enumerate}
\item 
Pour toute racine~$a$ de~$\phi$, 
$\log(\abs a)$ est une pente de~$\nu_\phi$.

\item 
Si $\phi$ est scindé,
il existe une numérotation $\{a_1,\dots,a_d\}$
de ses racines telle que pour tout entier $j\in\{1,\dots,d\}$,
la pente de~$\nu_\phi$ sur l'intervalle~$[j-1;j]$
soit égale à $\log(\abs{a_j})$.

\item Si $F$ est complet et $\phi$ est irréductible,
le polygone de Newton de~$\phi$ est affine,
de pente $\log(\abs{c_d}/\abs{c_0})/d$.
\end{enumerate}
\end{theo}
\begin{proof}
Notons $\phi=\sum_{m=0}^d c_m T^m$. Par hypothèse, $c_0\neq 0$
et $c_d\neq0$, de sorte que $\Pi_\phi=[0;d]$.

\begin{enumerate}
\item
Soit $a$ une racine de~$\phi$
et soit $M$ l'ensemble des entiers~$m\in\{0,\dots,d\}$
tels que $\abs{c_m} \abs a ^m $ soit maximal,
c'est-à-dire $\tau_\phi(\log(\abs a))=\log(\abs{c_m})+m\log(\abs a)$.
Puisque $\phi(a)=\sum_{m=0}^d c_m a^m=0$,
l'ensemble~$M$ possède au moins deux éléments.
La fonction $h\colon t \mapsto  t\log(\abs a)-\tau_\phi(\log(\abs a))$
est affine ; on a $h(m)=-\log(\abs{c_m})$ pour $m\in M$,
et $h(m)<-\log(\abs{c_m})$ sinon. 
Puisque $h$ est affine, elle est convexe et la définition de~$\nu_\phi$
entraîne l'inégalité $h\leq \nu_\phi$. 
Puisque $\nu_\phi-h$ est positive et convexe,
l'ensemble de ses zéros est convexe : les fonctions~$h$ et~$\nu_\phi$
coïncident sur l'enveloppe convexe de~$M$.
Cela prouve que $\log(\abs a)$ est une pente de~$\nu_\phi$.

\item
La fonction~$\nu_\phi$ est convexe, affine par morceaux;
ses pentes sont donc en ordre croissant.
Ordonnons les racines de~$\phi$ dans l'ordre de leurs valeurs absolues,
disons $a_1,\dots,a_d$.
Soit $h$ la fonction de~$[0;d]$ dans~$\R$
qui est affine sur chaque intervalle~$[m-1;m]$
et telle que 
\[ h(m) = - \log(\abs{a_{d-m+1}})- \dots - \log(\abs{a_d})
-\log(\abs{c_d}) \]
pour tout $m\in\{1,\dots,d\}$.
Sa pente sur l'intervalle~$[m-1;m]$ est 
$-\log(\abs{a_{d-m}})$;
ses pentes sont donc croissantes, de sorte que $h$~est convexe.

Les coefficients de~$\phi$ s'expriment comme les fonctions symétriques
élémentaires des~$a_i$: pour tout entier~$m$
tel que $0\leq m\leq d$, on a 
\[ c_m=(-1)^{d-m} c_d \sum_{i_1<\dots<i_{d-m}} a_{i_1}\dots a_{i_{d-m}}. \]
Dans cette expression, le terme de plus grande valeur absolue
est obtenu pour $(i_1,\dots,i_{d-m})=(d-m+1,\dots,d)$, 
de sorte que 
\[ \abs{c_{m}}  \leq \abs{c_d} \abs{a_{d-m+1}}\dots \abs{a_{d}} \]
pour tout entier $m\in\{0,\dots,d\}$,
c'est-à-dire $h(m) \leq -\log(\abs{c_m})$.
Puisque la fonction~$h$ est convexe et que
$\nu_\phi$ est la plus grande fonction convexe semi-continue
inférieurement qui vaut au plus $-\log(\abs{c_m})$ en~$m$,  
on a $h(t)\leq \nu_\phi(t)$ pour tout $t\in[0;d]$.

Soit $\Sigma$ l'ensemble des entiers~$m\in\{0,\dots,d\}$
tels que $m=d$ ou $\abs{a_{d-m}}<\abs{a_{d-m+1}}$.
Si $m\in\Sigma$,
il n'y a qu'un seul terme, dans l'expression de~$c_{m}$ donnée ci-dessus,
qui soit de valeur absolue maximale, si bien que l'on a
en fait $h(m)=-\log (\abs{c_m})$.  On a donc $h(m)=\nu_\phi(m)$.

Sur un intervalle~$[m,p]$  dont les extrémités
sont des éléments consécutifs de~$\Sigma$,
la fonction $\psi$ est affine, donc la fonction $\nu_\phi-h$ est 
convexe, semi-continue inférieurement et positive; 
comme elle s'annule à ses extrémités,
elle est constante sur cet intervalle.
Cela prouve que $\nu_\phi = h$ et conclut la démonstration
de~(2).

\item
Supposons maintenant que $F$ soit complet. 
Nous allons donner un argument par la théorie de Galois
en nous restreignant au cas où $\phi$ est séparable
(le cas général s'en déduit en écrivant $\phi=\psi(T^q)$,
où $q$ est une puissance de la caractéristique de~$F$
et $\psi$ est un polynôme séparable).
Soit $E$ une extension algébrique finie de~$F$, galoisienne, dans laquelle le
polynôme~$\phi$ est scindé; munissons-la de l'unique valeur
absolue qui prolonge celle de~$F$. Soit $G$ le groupe
de Galois de l'extension~$E/F$; c'est le groupe (fini)
des automorphismes de corps de~$E$ qui sont l'identité sur~$F$.
Pour tout $\sigma\in G$,
l'application $x\mapsto \abs{\sigma(x)}$ de~$E$ dans~$\R_+$
est une valeur absolue sur~$E$ dont la restriction à~$F$
est la valeur absolue initiale. 
Comme $F$ est complet, la valeur absolue de~$E$
est la seule qui prolonge celle de~$E$,
si bien que $\abs{\sigma(x)}=\abs{x}$ pour tout $x\in E$.

Par ailleurs, le groupe~$G$ stabilise 
l'ensemble $\{a_1,\dots,a_d\}$ des racines de~$\phi$;
comme $\phi$ est irréductible, l'action qui en résulte
est transitive.  Par suite, toutes les racines de~$\phi$
ont même valeur absolue.
En appliquant l'analyse de~(2) au corps~$E$,
il en résulte que le polygone de Newton de~$\phi$ n'a qu'une seule
pente. Comme il vaut~$\log(\abs{c_0})$ en~$0$ et
$\log(\abs{c_d})$ en~$d$, la formule annoncée en résulte.
\qedhere
\end{enumerate}
\end{proof}

\begin{coro}
Soit $F$ un corps, soit $m$ une valeur absolue ultramétrique non triviale
sur~$F$, dont la valuation est discrète. Soit $\pi$ un générateur
de l'idéal maximal de l'anneau de valuation~$A$ de~$F$.
Soit $\phi\in F[T]$ un polynôme unitaire; soit $d$ son degré.
Supposons que son polygone de Newton soit affine,
de pente~$(r/d)\log(\abs\pi)$, où $r$ et~$d$ sont premiers entre eux.
Alors $\phi$ est irréductible.
\end{coro}
\begin{proof}
On peut supposer que $F$ est complet.
Soit $\psi$ un facteur irréductible unitaire de~$\phi$.
Puisque les pentes de son polygone de Newton sont de la forme~$\log(\abs a)$,
où $a$ est une racine de~$\psi$, donc de~$\phi$, 
le polygone de Newton de~$\psi$ est affine, et sa pente
est $(r/d)\log(\abs\pi)$.
Posons $\psi=T^e+c_{e-1}T^{e-1}+\dots+c_0$;
on a nécessairement $\log(\abs{c_0})= e(r/d)\log(\abs\pi)$.
Nécessairement, $d$ divise~$er$. Comme $r$ et~$d$ sont premiers entre eux,
$d$ divise~$e$, d'où $d=e$. Ainsi, $\psi=\phi$.
\end{proof}

\begin{exem}
Un cas particulier du corollaire est le \emph{critère d'Eisenstein}
qui correspond au cas d'un polynôme 
$\phi=T^d+c_{d-1}T^{d-1}+\dots+c_0$ tel que
$c_0,\dots,c_{d-1}$ appartiennent à l'idéal maximal~$M$ de l'anneau
de valuation de~$F$,
et tel que $c_0$  n'appartienne pas à~$M^2$.
Avec les notations du corollaire, le polygone de Newton de~$\phi$
est affine, à pente $(1/d)\log(\abs\pi)$; le polynôme~$\phi$
est donc irréductible.

Pour un exemple spécifique et très classique, 
soit $p$ un nombre premier impair et posons
\[ \phi = \frac{(T+1)^{p}-1}{T}=T^{p-1} + \binom p{p-1} T^{p-2}+\dots + \binom p2T+\binom p1. \]
Ses racines complexes sont de la forme $-1+\omega$, où $\omega^p=1$ et $\omega\neq1$; autrement dit, 
il est égal à $\Phi_p(T+1)$, où $\Phi_p$ est le polynôme cyclotomique d'indice~$p$.
Le coefficient dominant de~$\phi$
vaut~$1$, et comme $p$ est premier, 
tous les autres coefficients sont multiples de~$p$;
le coefficient constant, égal à~$p$, n'est pas multiple de~$p^2$.
En appliquant le critère d'Eisenstein au corps~$\Q$,
muni de la valeur absolue $p$-adique, on en déduit que $\phi$ est irréductible.
En fait, il est même irréductible comme élément de~$\Q_p[T]$. 
\end{exem}

\begin{exem}\label{exem.schur}
Le critère d'Eisenstein montre comment déduire du polygone de Newton
un critère d'irréductibilité « local ». \citet{Coleman-1987} a expliqué
comment combiner plusieurs valeurs absolues pour en déduire 
des résultats d'irréductibilité « globaux ».  Il donne en particulier
l'exemple du polynôme 
\[ \phi = 1+ T + \dfrac12T^2+\dots + \dfrac1{n!}T^n \]
à coefficients dans~$\Q$.

Fixons un nombre premier~$p$, munissons~$\Q$ de la valeur absolue $p$-adique
et calculons son polygone de Newton.
C'est l'enveloppe convexe supérieure de l'ensemble des points
de la forme $(m, -\log(\abs{1/m!}_p))=(m,-v_p(m!)\log(p))$,
pour $m\in\{0,\dots,n\}$.

Écrivons le développement en base~$p$ de~$n$:
\[ n = a_mp^m+a_{m-1}p^{m-1}+\dots+a_0, \]
où $a_0,\dots,a_m\in\{0,\dots,p-1\}$.
Soit $(m_1,\dots,m_s)$ la suite strictement décroissante
des entiers~$k\in\{0,\dots,m\}$ tel que $a_k \neq0$;
pour tout~$i$, posons $n_i = a_{m_1}p^{m_1}+\dots+a_{m_i}p^{m_i}$.
Démontrons que les sommets de~$\tau_\phi$ sont les
points de la forme $(n_i, -v_p(n_i!)\log(p))$,
pour $i\in\{1,\dots,s\}$.

Rappelons d'abord comment évaluer la valuation $p$-adique de~$n!$:
parmi tous les entiers entre~$1$ et~$n$, $\lfloor n/p\rfloor$ sont multiples de~$p$, $\lfloor n/p^2\rfloor$ sont multiples de~$p^2$, etc., de sorte que
\[ v_p(n!) = \sum_{k=1}^\infty \lfloor n/p^k \rfloor. \]
Alors, $\lfloor n/p^k\rfloor=a_m p^{m-k}+\dots+a_k$ pour $k\leq m$, 
et est nul pour $k>m$,
si bien que
\[ v_p(n!) = \sum_{k=1}^m \sum_{j=k}^m a_j p^{j-k}
 = \sum_{j=1}^m a_j \sum_{k=1}^j p^{j-k}
= \sum_{j=1}^m a_j \frac{p^j-1}{p-1}
 = \frac {n - (a_0+\dots+a_m)}{p-1}. \]
En notant $\sigma(n)=a_0+\dots+a_m$;
on a donc $v_p(n!)=(n-\sigma(n))/(p-1)$.

Par définition, $\nu_\phi$ est la plus grande fonction convexe
semi-continue inférieurement
telle que $\nu_\phi(m)\leq -v_p(m!)\log(p)$ pour tout $m\in\{0,\dots,n\}$.
Par suite, $h\colon t\mapsto \nu_\phi(t)(p-1)/\log(p)+t$ 
est la plus grande fonction convexe semi-continue inférieurement
telle que $h(m)\leq \sigma(m)$
pour tout $m\in\{0,\dots,n\}$.

Comme la fonction nulle vérifie ces inégalités, on a $h(t)\geq 0$
pour tout~$t$.
Par suite, $h(0)=0$.
Si $(a-1)p^{m_1}<m\leq ap^{m_1}$, on a $\sigma(m)\geq a$,
et $\sigma(m)\geq a+1$ si $a_1p^{m_1}<m\leq n$.
Par suite, la fonction donnée par $t\mapsto t/a_1p^{m_1}$
vérifie les inégalités voulues; cela prouve
l'inégalité $t/p^{m_1}\leq h(t)$ pour tout $t\in[0;n]$.
Comme $h(a_1p^{m_1})\leq \sigma(a_1p^{m_1})=a_1$,
on a donc $h(a_1p^{m_1})=a_1$ et $h(t)=t/p^{m_1}$
sur l'intervalle $[0;a_1p^{m_1}]$.

On raisonne maintenant de même sur l'intervalle
$[a_1p^{m_1},n]$.
Pour tout entier~$m$ dans cet intervalle,
on a $\sigma(m)=1+\sigma(m-a_1p^{m_1})$.
Posant $h_1(t)=h(t+a_1p^{m_1})-a_1$,
on constate que~$h_1$ est à l'entier~$n-a_1p^{m_1}$
ce que~$h$ est à l'entier~$n$.
Cela prouve que $h$ est la fonction affine par morceaux
de pente $1/p^{m_1}$ sur l'intervalle $[0;n_1]$, 
de pente $1/p^{m_2}$ sur l'intervalle $[n_1;n_2]$,
\dots, de pente $1/p^{m_s}$ 
sur l'intervalle $[n_{s-1};n]$. Les sommets de~$h$
sont donc les entiers~$n_i$, et ceux de~$\nu_\phi$ aussi.

Les pentes de~$\nu_\phi$ sont donc les nombres réels
\[ \left(\dfrac 1{p^{m_i} (p-1)}-\dfrac1{p-1} \right)\log(p)
 = - \frac{p^{m_i}-1}{p^{m_i}(p-1)} \log(p) . \]

\begin{figure}[htbp]
\centering
\input{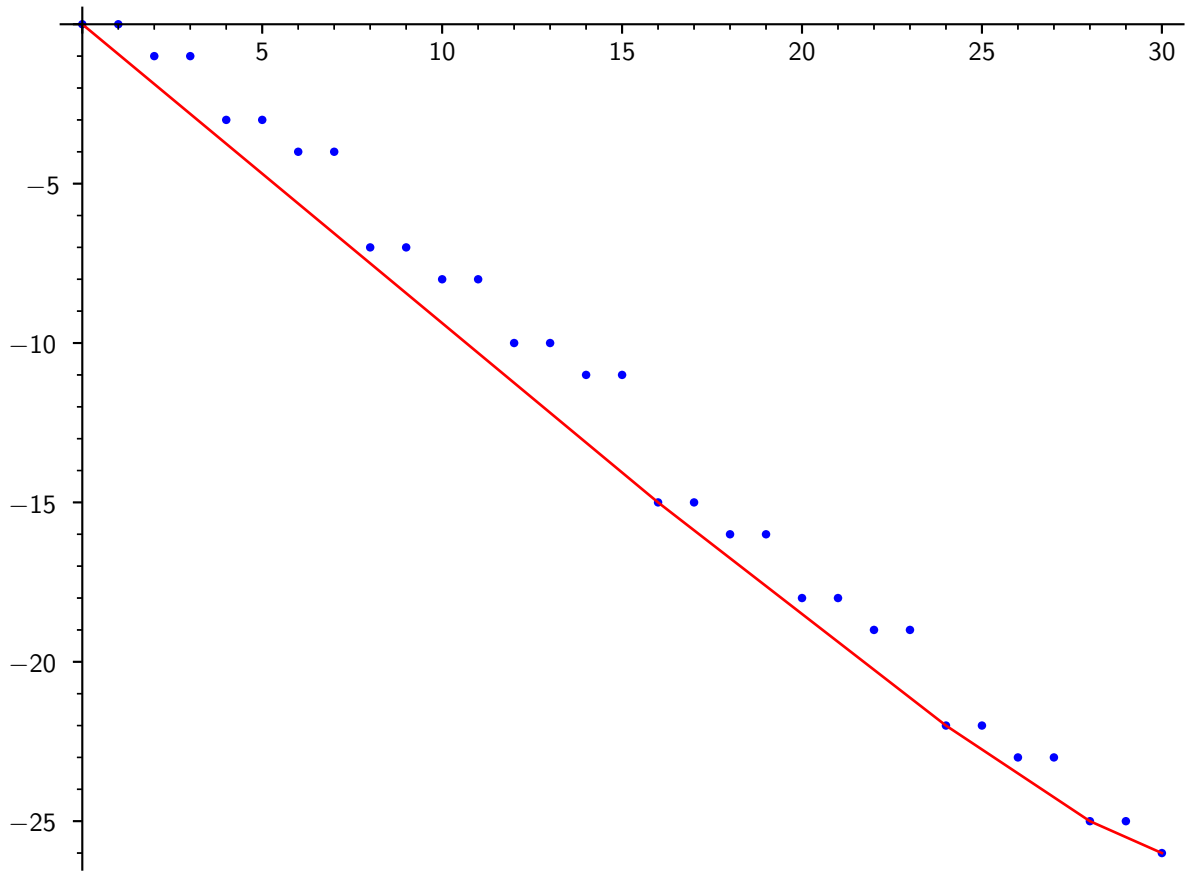}
\caption{Polygone de Newton du polynôme $\sum_{n=0}^{30}T^n/n!$
pour la valeur absolue $2$-adique. Les 4 pentes correspondent
à la décomposition $30=2^4+2^3+2^2+2$}
\end{figure}

Démontrons que $\phi$ est irréductible dans $\Q[T]$.
Considérons un facteur irréductible~$\psi$ de~$\phi$;
il s'agit de prouver que $\deg(\psi)=\deg(\phi)$.

Soit~$p$ un nombre premier qui divise~$n$ et soit $\psi_1\in\Q_p[T]$
un facteur irréductible de~$\psi$.
Comme $\Q_p$ est complet, le théorème~\ref{theo.newton-pol-pol}, (3),
entraîne que le polygone de Newton (pour la valeur
absolue $p$-adique) de~$\psi_1$ n'a qu'une seule pente,
et cette pente~$\nu$ est de la forme $k \log(p)/\deg(\psi_1)$,
pour un certain entier relatif~$k$.
D'autre part, $\nu$ apparaît dans les pentes du polygone de Newton de~$\phi$,
de sorte qu'il existe $i\in\{1,\dots,s\}$
tel que $kp^{m_i}=-\deg(\psi_1)(p^{m_i}-1)/(p-1)$.
En particulier, $p^{m_i}$ divise~$\deg(\psi_1)$. 
Or, avec les notations précédentes, on $m_s=v_p(n)$,
si bien que $p^{v_p(n)}$ divise~$\deg(\psi_1)$.
Comme c'est le cas de chaque facteur irréductible~$\psi_1$
de~$\psi$, il s'ensuit que $p^{v_p(n)}$ divise~$\deg(\psi)$.

Comme ceci vaut pour tout nombre premier~$p$ qui divise~$n$,
on en conclut que $\deg(\phi)=n=\prod_p p^{v_p(n)}$ divise~$\deg(\psi)$.
Par conséquent, $\deg(\psi)=\deg(\phi)$, et cela conclut
la preuve que $\phi$ est irréductible.

En précisant cet argument, \citet{Coleman-1987}
redémontre un théorème de~\citet{Schur-1930}
selon lequel le groupe de Galois du polynôme~$\phi$
est le groupe alterné~$\mathfrak A_n$ si $n$~est multiple de~$4$,
et le groupe symétrique sinon.
\end{exem}

\section{Fonctions et équations analytiques}

Soit $F$ un corps muni d'une valeur absolue 
pour laquelle il est complet.
On peut développer une théorie des séries
entières dans~$F$ dont les débuts sont tout à fait parallèles
à celle des séries entières à coefficients complexes.
Pour simplifier, on n'expose que le cas d'une seule indéterminée.

\subsection{}
Soit $\phi = \sum_{n=0}^\infty c_n T^n$
une série formelle en une indéterminée~$T$ à coefficients dans~$F$.
On définit son \emph{rayon de convergence} 
$R_\phi\in[0;+\infty]$
par la formule usuelle:
\[ R_\phi = \left(\limsup_{n} \abs {c_n}^{1/n}\right)^{-1}. \]
Par définition, si $\abs a<R_\phi$, alors $c_n a^n$ tend vers~$0$,
tandis que si $\abs a>R_\phi$, alors $c_n a^n$ n'est pas bornée.

Le cas de la valeur absolue triviale est peu intéressant:
on a $R_\phi=1$  si $\phi\neq 0$, et pour $a\neq 0$,
la série $\sum c_na^n$ ne converge que si $\phi$ est un polynôme.
Le cas d'une valeur absolue archimédienne est probablement bien connu;
on supposera donc dans la suite que la valeur absolue est
ultramétrique et non triviale.

\subsection{}
Pour toute série formelle $\phi\in F\lbra T\rbra$, disons $\phi=\sum c_nT^n$, 
et tout nombre réel~$R\geq0$, posons 
$\norm\phi_R=\sup_n \abs{c_n}R^n$; c'est un élément de~$[0;+\infty]$.
C'est une fonction croissante de~$R$, continue à gauche :
pour $R>0$, on a 
\[ \norm\phi_R = \sup_{r<R} \norm\phi_r.\]
Soit en effet $A$ tel que $A<\norm\phi_R$;
il existe donc $n\geq 0$ tel que $A<\abs{c_n}R^n\leq \norm\phi_R$;
il existe alors $r$ tel que $r<R$ et $A<\abs{c_n}r^n$,
d'où $A<\norm\phi_r\leq \norm\phi_R$.
\subsection{}
Pour $a\in F$ et $\phi,\psi\in F\lbra T\rbra$ tels que $\norm\phi_R,\norm\psi_R<\infty$, on a les relations :
\begin{enumerate}
\item Si $R>0$, alors $\norm\phi_R=0$ si et seulement si $\phi=0$
(pour $R=0$, on a $\norm\phi_0=\abs{c_0}$);
\item On a $\norm{\phi+\psi}_R\leq \sup(\norm\phi_R,\norm\psi_R)$;
\item Avec la convention $0\cdot\infty=0$, on a 
$\norm{a\phi}_R=\abs a\norm\phi_R$ 
et $\norm{\phi\psi}=\norm\phi\norm\psi$.
\end{enumerate}
Cela démontre que l'ensemble des séries~$\phi$ telles que $\norm\phi_R<+\infty$
est une sous-algèbre de $F\lbra T\rbra$ sur laquelle $\norm{\,\cdot\,}_R$
est une norme et une valeur absolue.
Seule la propriété de multiplicativité n'est pas immédiate:
posons $\phi=\sum a_n T^n$ et $\psi=\sum b_nT^n$;
alors, $\phi\psi=\sum_n c_n T^n$, où
$ c_n = \sum_{k=0}^n a_{n-k}b_k$ pour tout~$k$.
Par hypothèse, on a $\abs{a_{n-k}}R^{n-k}\leq \norm\phi_R$
et $\abs{b_k}R^k\leq \norm\psi_R$  pour tous $k,n\in\N$ tels que $k\leq n$;
par l'inégalité ultramétrique, on en déduit 
\[ \abs{c_n}R^n\leq \sup_k (\abs{a_{n-k}}\abs{b_k}R^n)
\leq \norm\phi_R\norm\psi_R, \]
de sorte que $\norm{\phi\psi}_R\leq \norm\phi_R\norm\psi_R$.
Pour démontrer l'autre inégalité, considérons des nombres réels~$\alpha$
et~$\beta$ tels que 
$0\leq \alpha<\norm\phi_R$  et $0\leq \beta<\norm\psi_R$.
Soit alors $r\in\R$ tel que $0\leq r<R$ et $\alpha<\norm\phi_r$ et $\beta<\norm\psi_r$. 
Comme $0\leq r<R$, 
les suites $\abs{a_n}r^n$ et $\abs{b_n}r^n$ tendent vers~$0$,
de sorte qu'il existe un plus grand entier $m\in\N$
tel que $\abs{a_m}r^m=\norm\phi_r$ et un plus grand entier $p\in\N$
tel que $\abs{b_p}r^p=\norm\psi_r$.
Posons $n=m+p$;
dans l'expression $c_n=\sum_{k=0}^n a_{n-k}b_k$, le terme $a_mb_p$
est le seul qui soit de valeur absolue maximale, 
de sorte que $\abs{c_n}r^n=\abs{a_m}r^m\abs{b_p}r^p=\norm\phi_r\norm\psi_r$.
Alors, $\norm{\phi\psi}_r=\norm\phi_r\norm\psi_r>\alpha\beta$.
A fortiori, $\norm{\phi\psi}_R\geq \alpha\beta$,
et l'inégalité $\norm{\phi\psi}_R\geq \norm\phi_R\norm\psi_R$ 
s'en déduit en faisant tendre~$\alpha$ et~$\beta$
vers $\norm\phi_R$ et $\norm\psi_R$ respectivement.

\begin{lemm}
Soit $R$ un nombre réel strictement positif.
Les séries $\sum c_n T^n$ telles que $(\abs{c_n}R^n)$
tende vers~$0$ forment une sous-algèbre de $F\lbra T\rbra$
dont les éléments ont tous rayon de convergence au moins~$R$.

La restriction de~$\norm{\,\cdot\,}_R$ à cette algèbre
est une norme et une valeur absolue pour laquelle elle est complète.
\end{lemm}
Lorsque $R=1$, cette algèbre est notée $F\langle T\rangle$
et est appelée \emph{algèbre de Tate} en l'indéterminée~$T$;
ses éléments sont aussi appelés \emph{séries formelles restreintes}.
La définition générale de cette algèbre semble due à~\citet{Berkovich-1990};
nous l'appellerons l'\emph{algèbre de Tate} de rayon~$R$ et la noterons
$F\langle T/R\rangle$. 

\begin{proof}
Les démonstrations se déduisent pour l'essentiel
de ce qui précède.  Démontrons la stabilité par produit.
Si $\phi=\sum a_n T^n$, $\psi=\sum b_nT^n$ et $\phi\psi=\sum c_nT^n$,
on a $ c_n = \sum_{p+q=n} a_p b_q $
pour tout~$n$;
démontrons que $\abs{c_n}R^n$ tend vers~$0$.
Soit $\eps>0$; soit $N\geq 0$
tel que $\abs{a_n}R^n\leq \eps$ 
et $\abs{b_n}R^n\leq \eps$ pour $n\geq N$.
Alors, si $n\geq 2N$, toute décomposition $n=p+q$
entraîne $p\geq N$ ou $q\geq N$;
dans le premier cas, $\abs{a_p}R^p\abs{b_q}R^q\leq \eps \norm \psi$,
et dans le second, $\abs{a_p}R^p\abs{b_q}R^q\leq \eps\norm\phi$.
Ainsi, $\abs{c_n}R^n\leq \eps\sup(\norm\phi,\norm\psi)$.
Cela prouve que $\abs{c_n}R^n$ tend vers~$0$.

Cette algèbre est complète:
soit $(\phi_m)$ une suite de Cauchy dans~$F\langle T/R\rangle$;
pour tout~$m$, posons $\phi_m=\sum c_{m,n}T^n$.

Soit $n$ un entier.
On a  $\abs {c_{m,n}-c_{p,n}}R^n\leq \norm{\phi_m-\phi_p}$,
pour tous~$m,p$,
ce qui prouve que  la suite $(c_{m,n})_m$
vérifie le critère de Cauchy. Puisque~$F$ est complet,
elle converge dans~$F$; notons~$c_n$ sa limite.
Soit $\phi$ la série formelle $\sum c_n T^n$;
démontrons qu'elle appartient à~$F\langle T/R\rangle$
et que $(\phi_m)$ converge vers~$\phi$ dans~$F\langle T/R\rangle$.

Soit $\eps>0$ et soit $M$ un entier tel que $\norm{\phi_m-\phi_p}\leq\eps$
pour $m,p\geq M$.
Fixons~$m\geq M$; en faisant tendre~$p$ vers l'infini,
on obtient $\abs{c_{m,n}-c_n}R^n\leq \eps$ pour tout~$n$.
Puisque $(\abs{c_{m,n}}R^n)_m$ tend vers~$0$, il existe 
un entier~$N$
tel que $\abs{c_{m,n}}R^n\leq \eps$ pour tout $n\geq N$.
Alors $\abs{c_n}R^n \leq \sup(\abs{c_{m,n}-c_n},\abs{c_{m,n}})R^n\leq\eps$
pour $n\geq N$. 
Cela prouve que $(\abs{c_n}R^n)$ tend vers~$0$,
donc $\phi\in F\langle T/R\rangle$.
Les mêmes inégalités prouvent que $\norm{\phi_m-\phi}\leq\eps$
pour $m\geq M$. Cela démontre que $(\phi_m)$ tend vers~$\phi$.
\end{proof}

\subsection{}
Soit $\phi\in F\langle T/R\rangle$; notons $\phi=\sum c_nT^n$.
Soit $a\in F$ tel que $\abs a \leq R$.
Par hypothèse, la suite de terme général~$c_n a^n$ tend vers~$0$;
Comme $F$ est complet, la série $\sum c_n a^n$ converge ;
on note $\phi(a)$ sa somme.
Il y a même convergence normale, de sorte que
l'application $a\mapsto \phi(a)$ ainsi définie
est continue sur la boule circonférenciée $B(0,R)$.

De plus, l'application $\phi\mapsto \phi(a)$ est un morphisme d'algèbres
de $F\langle T/R\rangle $ dans~$F$.

Il y a également un résultat pour la composition des séries
formelles: soit $\phi\in F\lbra T\rbra$
et soit $\psi\in F\langle T/S\rangle$. Notons $\phi=\sum_{m=0}^\infty c_mT^m$.
Si la suite $(\abs{c_m} \norm\psi_S^m)$ tend vers~$0$,
la série $\sum c_m \psi^m$ converge dans l'algèbre complète
$F\langle T/S\rangle$.  On note $\phi\circ\psi$ sa limite.
C'est en particulier le cas si $\norm\psi_S<R_\phi$,
ou si $\phi\in F\langle T/R\rangle$ et $\norm\psi_S\leq R$.

Notons deux cas particuliers importants:
\begin{enumerate}
\item Soit $a\in F$;  supposons $\phi\in F\langle T/R\rangle$ et $\abs a \leq R$. Alors, on peut former la série formelle $\phi_a=\phi(a+T)$;
elle appartient à $F\langle T/R\rangle$; pour tout $b\in B(0;R)$,
on a $\phi_a(b)=\phi(a+b)$. 
Les coefficients de~$\phi_a$ sont donnés par la « formule de Taylor ».
Il existe une unique suite $(\phi^{[n]})$ d'éléments de~$F\langle T/R\rangle$
telle que 
\[ \phi_a(T) = \sum_{n=0}^\infty \phi^{[n]} (a) T^n. \]
Explicitement, si $\phi=\sum c_mT^m$, on a 
\[ \phi^{[n]} = \sum_{m=0}^\infty c_{m+n} (m+n)\dots (m+1) T^n. \]
Lorsque $F$ est de caractéristique zéro, $\phi^{[n]}$ est reliée
à la $n$-ième dérivée formelle de~$\phi$ 
par la formule $\phi^{[n]}=\phi^{(n)}/n!$.
En particulier,  $\phi$ est indéfiniment dérivable sur $B(0;R)$,
de dérivée~$\phi^{[1]}$.

\item Le rayon de convergence de la série formelle $1/(1-T)$ est égal à~$1$.
Ainsi, toute série $\psi\in F\langle T/S\rangle$
telle que $\norm{\psi-1}_S<1$ est inversible dans~$F\langle T/S\rangle$.
Nous déterminerons ci-dessous les éléments inversibles de cette algèbre.
\end{enumerate}


\begin{prop}\label{prop.tate-inversible}
Soit $\phi\in F\langle T/R\rangle$; notons $\phi=\sum_{n=0}^\infty c_nT^n$.
Les conditions suivantes sont équivalentes:
\begin{enumerate}
\item La série~$\phi$ est inversible dans $F\langle T/R\rangle$;
\item Pour tout entier~$n>0$, on a $\abs{c_n}R^n<\abs{c_0}$;
\item On a $\norm{\phi-\phi(0)}_R <\abs{\phi(0)}$.
\end{enumerate}
\end{prop}
\begin{proof}
Supposons la condition~(3) satisfaite et démontrons
que $\phi$ est inversible.
Quitte à remplacer~$\phi$ par~$\phi/\phi(0)$, on se ramène
au cas où $\phi(0)=1$ et où $\norm{1-\phi}_R<1$.
On peut alors considérer la série convergente $\sum (1-\phi)^n$ dans l'algèbre
complète $F\langle T/R\rangle$, et sa limite~$\psi$
vérifie $(1-\phi)\psi = \psi - 1$, c'est donc l'inverse de~$\phi$.

Supposons maintenant que~$\phi$ soit inversible dans~$F\langle T/R\rangle$
et soit $\psi$ son inverse.
Posons $\phi=\sum_{n=0}^\infty a_nT^n$, $\psi=\sum_{n=0}^\infty b_nT^n$
et $\phi\psi=\sum_{n=0}^\infty c_nT^n$ (on a donc $c_0=1$ et $c_n=0$ si $n>0$).
Soit $m$ le plus grand entier tel que $\abs{a_m}R^m=\norm{\phi}_R$
et soit $n$ le plus grand entier tel que $\abs{b_n}R^n=\norm{\psi}_R$.
Dans l'expression 
$ c_{n+m}=\sum_{k=0}^{m+n} a_{m+n-k} b_k $, 
le terme $a_m b_n$ est le seul de valeur absolue maximale,
si bien que $\abs{c_{n+m}}=\abs{a_m}\abs{b_n}$.
Puisque $c_k=0$ pour $k>0$, cela entraîne $m=n=0$.
On a donc $\abs{a_k}R^k<\abs{a_0}$ pour tout entier~$k>0$,
ce qui prouve~(2).

Supposons enfin que l'assertion~(2) est vérifiée.
Comme la suite $(\abs{a_k}R^k)$ tend vers~$0$,
il existe un plus entier~$n>0$ tel que $\abs{a_n}R^n=\sup_{k>0} \abs{a_k}R^k
=\norm{\phi-a_0}_R$.
Puisque $\abs{a_n}R^n<\abs{a_0}$ par hypothèse et $a_0=\phi(0)$, il vient donc
$\norm{\phi-\phi(0)}_R<\abs{\phi(0)}$.
\end{proof}

\subsection{}
La théorie des polygones de Newton admet une variante pour
les fonctions analytiques d'une variable.
Soit $\phi\in F\lbra T\rbra$; notons-la $\phi=\sum c_nT^n$.

Pour tout $x\in\R$, on pose 
\[ \tau_\phi(x)=\log(\norm\phi_{e^x}) = \sup_n \log (\abs{c_n})+n x , \]
formule dans laquelle on convient d'exclure les entiers~$n$
tels que $c_n=0$, ou bien l'on pose $\log(\abs 0)=-\infty$…
Si le rayon de convergence de~$\phi$ est nul,
alors $\tau_\phi$ est constante, égale à~$+\infty$.
Si $\phi=0$, alors $\tau_\phi$ est constante, égale à~$-\infty$.
Dans la suite de la discussion,
on exclut implicitement ces deux cas triviaux.

Soit $d=\ord(\phi)$ le plus petit entier naturel tel que $c_d\neq 0$.
Lorsque $x$ tend vers~$-\infty$, on a $\tau_\phi(x)=\log(\abs{c_d})+d x$.

Lorsque $\abs{c_n} e^{r^n}$ tend vers~$0$, 
en particulier lorsque $r< \log(R_\phi)$,
ou bien lorsque $r\leq \log(R)$ et $\phi\in F\langle T/R\rangle$,
il existe un plus grand entier~$N$ 
tel que $\log(\abs{c_n})+n r =\tau_\phi(r)$.
Alors, il suffit de considérer
les entiers~$n\leq N$ 
dans la borne supérieure qui définit $\tau_\phi(r)$,
si bien que $\nu_\phi$ 
est une fonction convexe, affine par morceaux,
sur l'intervalle~$\mathopen]-\infty;r]$.
La fonction~$\tau_\phi$ est croissante, semi-continue
inférieurement, continue à gauche, convexe sur~$\R$, 
à valeurs dans~$\R\cup\{+\infty\}$.

Soit $x\in\R$ et soit $R=e^x$; on suppose que $\phi\in F\langle T/R\rangle$.
Pour tout $a\in F$ tel que $\abs a=R$, on peut donc évaluer $\phi(a)$.
Par l'inégalité ultramétrique, on a $\log(\abs{\phi(a)})\leq \tau_\phi(a)$.

Comme $\abs{c_n}R^n$ tend vers~$0$ et $x\leq\log(R)$, l'ensemble~$N$
des entiers~$n$ tels que $\log(\abs{c_n})+nx=\tau_\phi(x)$
est fini et non vide. S'il est réduit à un seul élément~$n$,
on a alors $\abs{\phi(a)}=\abs{c_n} R^n$, c'est-à-dire 
$\log(\abs{\phi(a)})=\tau_\phi(a)$; 
en particulier, $\phi(a)\neq0$.
Comme dans le cas des polynômes,
les valeurs absolues des zéros de~$\phi$ sont donc liées aux
points anguleux de~$\tau_\phi$.

\subsection{}
On définit le polygone de Newton~$\nu_\phi$ de~$\phi$
comme la transformée de Legendre de~$\tau_\phi$: pour $t\in\R$, on pose
\[ \nu_\phi(t) = \sup_{x} \left( x t - \tau_\phi(x)\right). \]
Bien entendu, dans cette borne supérieure, il suffit de se borner
aux nombres réels~$x$ tels que $x < \log(R_\phi)$.

Soit $d=\ord(\phi)$.
Pour $x$ tendant vers~$-\infty$, on a $\tau_\phi(x)=xd+\log(\abs{c_d})$,
donc $\nu_\phi(t)\geq -(d-t)x -\log(\abs{c_d})$.
Ainsi, on a $\nu_\phi(t)=+\infty$ si $t<d$.
D'autre part, si on écrit $\phi=T^d\phi_1$,
on a $\tau_\phi=dx+\tau_{\phi_1}$,
si bien que $\nu_{\phi}(t)=\nu_{\phi_1}(t-d)$.
On pourra ainsi se ramener au cas où $\ord(\phi)=0$,
c'est-à-dire $\phi(0)\neq0$.

Si $\phi$ est un polynôme de degré~$n$,
on a de même $\tau_\phi(x)=nx+\log(\abs{c_n})$ pour $x$ tendant 
vers~$+\infty$, donc $\nu_\phi(t)=+\infty$ si $t>n$.
On supposera souvent implicitement 
dans la suite que $\phi$ n'est pas un polynôme.

Par dualité en théorie des fonctions convexes,
$\nu_\phi$ est la plus grande fonction convexe semi-continue
inférieurement sur~$\R$ (à valeurs
dans~$\R\cup\{+\infty\}$)
telle que $\nu_\phi(m)\leq -\log(\abs{c_m})$ pour tout $m\in\N$
tel que $c_m\neq0$. 
Son graphe est le bord de l'enveloppe convexe supérieure
des points $(m,-\log(\abs{c_m}))$. 

Cela donne un moyen pratique de la construire dans des cas concrets.
On part en effet de l'égalité $\nu_\phi(0)=-\log(\abs{c_0})$ 
qui sera justifiée ci-dessous.
Posons $p=\inf_{n>0} (- \log(\abs{c_n}) + \log(\abs{c_0}))/n$;
la fonction $h\colon t\mapsto -\log(\abs{c_0})+pt$ est la plus grande
fonction affine qui coïncide avec~$\nu_\phi$ en~$0$
et minore~$\nu_\phi$.
Soit $N$ l'ensemble des entiers~$n>0$
tels que la borne supérieure précédente soit atteinte en~$n$.
S'il est vide ou s'il est infini, alors $\nu_\phi=h$.
Supposons que $N$ soit non vide et fini et soit~$n$ son plus grand
élément. La fonction $\nu_\phi-h$ sur~$[0;n]$ est convexe,
semi-continue inférieurement,  positive, et vaut~$0$ en~$0$ et~$n$;
elle est donc nulle, si bien que $\nu_\phi$ coïncide avec~$h$
sur l'intervalle~$[0;n]$.
On reprend alors l'argument précédent considérant 
la plus grande fonction affine de pente~$\geq p$
qui minore~$\nu_\phi$ et coïncide avec~$\nu_\phi$ en~$n$, etc.

\begin{exem}\label{exem.exp-log}
Supposons 
que $F$ soit le corps des nombres rationnels, muni de la valeur
absolue $p$-adique. Pour simplifier les notations, on suppose
que le logarithme est pris en base~$p$, c'est-à-dire que $\log(p)=1$.
\begin{enumerate}
\item Considérons le cas où
$\phi=e^T=\sum_{n=0}^\infty T^n/n!$.

On a vu la relation $-\log(1/\abs{n!})=(\sigma(n)-n)/(p-1)$,
où $\sigma(n)$ est la somme des chiffres du développement
en base~$p$ de~$n$. 
Par suite, la fonction affine $t\mapsto -t/(p-1)$
est inférieure ou égale à $-\log(\abs{1/n!})$ en~$n$;
on a donc $\nu_\phi(t)\geq -t/(p-1)$ pour tout $t\in\R_+$.
Comme $\nu_\phi(0)\leq -\log(\abs{1})=0$, on  a 
aussi $\nu_\phi(0)=0$.
Prenons $n=p^m$; on a $\sigma(n)=1$, donc $\nu_\phi(n)\leq (1-n)/(p-1)$.
Par convexité, il vient 
$\nu_\phi(t)\leq t/n(p-1) - t/(p-1)$ si $t\in[0;n]$.
Faisons tendre~$m$ (donc~$n$) vers~$+\infty$; il en résulte
l'inégalité $\nu_\phi(t)\leq -t/(p-1)$.
Ainsi, le polygone de Newton de~$\phi$ est donné par $t\mapsto -t/(p-1)$.

\item
Prenons maintenant $\phi=\log(1+T)=\sum_{n=1}^\infty (-1)^{n-1}T^n/n$.
Puisque $\ord(\phi)=1$, son polygone de Newton est infini entre~$0$ et~$1$.
Démontrons que pour tout entier~$m\geq 0$, on a $\nu_\phi(p^m)=-m$
et que la restriction de~$\nu_\phi$ à l'intervalle~$[p^m;p^{m+1}]$
est affine.

Notons~$h$ la fonction affine par morceaux ainsi définie;
il est immédiat qu'elle est décroissante et convexe.
pour tout $m\in\N$, on a $h(p^m)=-m=-\log(\abs{1/p^m})$,
donc $h(p^m)\leq \nu_\phi(p^m)$.
Soit $n\in\N$ tel que $n\geq 1$;
soit $m$ l'unique entier naturel tel que $p^m\leq n< p^{m+1}$;
puisque $n< p^{m+1}$, on a $v_p(n)\leq m$, 
donc $-\log(\abs{1/n}) \geq -m = h(p^m)$;
a fortiori, $-\log(\abs{1/n})\geq -(m+1)=h(p^{m+1})$.
Sur l'intervalle~$[p^m;p^{m+1}]$, la fonction $h$ est affine,
inférieure ou égale à~$\nu_\phi$, et coïncide avec~$\nu_\phi$, 
aux extrémités. Puisque $\nu_\phi$ est convexe, 
ces deux fonctions coïncident sur cet intervalle.
\end{enumerate}
\end{exem}

\begin{theo}
Soit $R$ un nombre réel strictement positif 
et soit $\phi\in F\langle T/R\rangle$ une série formelle restreinte.
On suppose que $\phi$ n'est pas un polynôme et que $\phi(0)\neq 0$.
\begin{enumerate}
\item
La fonction~$\nu_\phi$ est définie sur~$\R_+$;
elle est continue, convexe et affine par morceaux.

\item
Les nombres réels~$t$ tels que $\nu_\phi$ soit finie et non dérivable en~$t$
sont des entiers tels que $\nu_\phi(t)=-\log(\abs{c_t})$.
Soit $(n_0,n_1,\dots)$ la suite strictement croissante 
de ces entiers. On a $n_0=0$.

\item
Soit $a\in F$ tel que $\abs a\leq R$ et $\phi(a)=0$; 
alors $\log(\abs a)$ est une pente de~$\nu_\phi$. 

\item 
Soit~$m$ un entier tel que~$n_m$ est défini et
soit $\mu$ la pente de~$\nu_\phi$ sur l'intervalle~$[n_{m-1};n_m]$.
On a $\mu < \log(R_\phi)$.
Si, de plus, $\mu\leq\log(R)$,
il existe un unique couple $(P,\psi)$ où $P\in F[T]$ est
un polynôme de degré~$n_m-n_{m-1}$ tel que $P(0)=1$
et dont le polygone de Newton n'a qu'une pente, égale à~$\mu$,
et où $\psi\in F\langle T/R\rangle$ est telle que $\phi=P\psi$.
De plus, $\mu$ n'est pas une pente de~$\nu_\psi$.
\end{enumerate}
\end{theo}
\begin{proof}
\begin{enumerate}
\item
Soit $n$ un entier naturel tel que $c_n\neq0$;
rappelons aussi que $c_0\neq0$, par hypothèse.
Toute fonction convexe qui est majorée par $-\log(\abs{c_0})$
en~$0$ et par $-\log(\abs{c_n})$ en~$n$
est majorée sur l'intervalle~$[0;n]$, de sorte  
que $\nu_\phi$ est finie sur cet intervalle. Puisque $\phi$
n'est pas un polynôme, il existe des entiers~$n$ arbitrairement grands
tels que $c_n\neq0$. Cela prouve que $\nu_\phi$ est finie sur~$\R_+$.
Étant convexe et semi-continue inférieurement,
elle est donc continue sur cet intervalle.

\item
Soit~$h$ l'unique fonction affine par morceaux
qui est affine sur chaque intervalle de la forme~$[n;n+1]$
et coïncide avec~$\nu_\phi$ sur cet intervalle.
Elle est convexe et continue, et vérifie $\nu_\phi\leq h$.
Par construction, on a $h(n)\leq -\log(\abs{c_n})$
pour tout~$n$; par suite, $h=\nu_\phi$. Cela prouve
que $\nu_\phi$ est affine sur chaque intervalle de la forme~$[n;n+1]$.

La fonction~$h$ qui coïncide avec~$\nu_\phi$ sur~$[1;+\infty]$,
est affine sur~$[0;1]$ et vaut~$-\log(\abs{c_0})$ en~$0$
est convexe, semi-continue inférieurement, et majore~$\nu_\phi$; 
elle coïncide donc avec~$\nu_\phi$. Par suite, $\nu_\phi(0)=-\log(\abs{c_0})$.

Soit $n$ un entier tel que $n\geq 1$ et $\nu_\phi(n) < -\log(\abs{c_n})$;
supposons que les dérivées à droite et à gauche de~$\nu_\phi$ en~$n$
soient distinctes. Pour $\eps>0$ assez petit,
l'unique fonction~$h$ sur~$\R_+$ qui est affine sur chaque intervalle
de la forme~$[m;m+1]$, 
coïncide avec~$\nu_\phi$ hors de~$[n-1;n+1]$
et qui  vaut $\nu_\phi(n)+\eps$ en~$n$
est encore convexe, continue, et vérifie $ h(m)\leq-\log(\abs{c_m})$
pour tout $m\in\N$. Cela contredit la définition de~$\nu_\phi$.

\item
Soit $x\in\R$ tel que $x\leq\log(R)$.
Pour tout $a\in F$ tel que $\abs a=e^x$, on peut donc évaluer $\phi(a)$.
Par l'inégalité ultramétrique, on a $\log(\abs{\phi(a)})\leq \tau_\phi(a)$.
Supposons $\phi(a)=0$.
Comme $\abs{c_n}R^n$ tend vers~$0$ et $x\leq\log(R)$, l'ensemble~$N$
des entiers~$n$ tels que $\log(\abs{c_n})+nx=\tau_\phi(x)$
est fini et non vide. 
Pour $n\in N$, on a 
\[ \nu_\phi(n) \geq n x - \tau_\phi(x) = -\log(\abs{c_n}), \]
de sorte que $\nu_\phi(n)=-\log(\abs{c_n})$.

Si l'ensemble~$N$ est réduit à un seul élément~$n$,
on a alors $\abs{\phi(a)}=\abs{c_n} e^{nx}$, c'est-à-dire 
$\log(\abs{\phi(a)})=\tau_\phi(a)$, ce qui contredit
l'hypothèse $\phi(a)=0$.

Notons alors~$n$ et~$n'$ le plus petit et le plus grand élément de~$N$,
de sorte que $n<n'$, et soit~$h$  l'unique fonction affine
qui coïncide avec~$\nu_\phi$ en~$n$ et en~$n'$.
Démontrons que $\nu_\phi$ et~$h$ coïncident
sur l'intervalle~$[n;n']$. Par convexité, on a $\nu_\phi\leq h$
sur cet intervalle et $\nu_\phi\geq h$ en dehors.
Il suffit donc de prouver que $\nu_\phi=\sup(h,\nu_\phi)$ et
donc que l'on a $h(m)\leq -\log(\abs{c_m})$ pour tout $m\in\N$;
on peut pour cela supposer que $m\in[n;n']$.
Écrivons $m=(1-u)n+un'$, où $u\in[0;1]$.
L'inégalité $\log(\abs{c_m})+mx\leq \tau_\phi(x)$
se récrit
\begin{align*}
 \log(\abs{c_m}) & \leq - mx+\tau_\phi(x)  \\
& = (1-u) (-nx+\tau_\phi(n))
+ u (-n'x+\tau_\phi(n')) \\
& = (1-u)(-h(n))+u (-h(n'))=-h(m) ,\end{align*}
de sorte que $h(m)\leq -\log(\abs{c_m})$, ainsi qu'il fallait démontrer.

Ainsi, $\nu_\phi$ est affine sur l'intervalle~$[n;n']$, de pente
\[ \frac{\nu_\phi(n')-\nu_\phi(n)}{n'-n}
 = \frac{ (n'x-\tau_\phi(x))-(nx-\tau_\phi(x))}{n'-n}=x. \]
Cela prouve que $x=\log(\abs a)$ est une pente de~$\nu_\phi$.

\item
Par hypothèse, 
les pentes de~$\nu_\phi$ sur l'intervalle~$[n_m;+\infty\mathclose[$
sont strictement supérieures à~$\mu$; il existe donc un nombre réel~$\mu'>\mu$
tel que 
\[ \nu_\phi(t) \geq \nu_\phi(n_m)+\mu'(t-n_m) \]
pour tout $t\geq n_m$. 
En particulier, pour tout entier~$n\geq n_m$, on a 
\[ -\log(\abs{c_n}) \geq \nu_\phi(n) \geq \nu_\phi(n_m)+\mu'(n-n_m), \]
si bien que
la suite $( \abs{c_n} e^{n\mu'} )$ est bornée, si
bien que $\mu< \mu'\leq \log(R_\phi)$.

Supposons $\mu\leq \log(R)$.
Pour tout entier~$N$, posons $\phi_N=\sum_{n=0}^N c_nT^n$.
Supposons $N> n_j$. Alors, la restriction du polygone de Newton de~$\phi_N$
à l'intervalle~$[0;n_m] $ coïncide avec celui de~$\phi$.
Sa restriction à l'intervalle~$[n_m;N]$ ne coïncide pas
forcément avec celui de~$\phi$, c'est  d'ailleurs ce
qu'indique la confrontation
des exemples~\ref{exem.schur} et~\ref{exem.exp-log}, 
mais le polygone de Newton de~$\phi_N$
domine celui de~$\phi$ sur cet intervalle.

On déduit du théorème~\ref{theo.newton-pol-pol}
qu'il existe un unique polynôme~$P_N\in F[T]$
de degré~$n_m-n_{m-1}$ tel que $P_N(0)=1$,
dont le polygone de Newton est affine de pente~$\mu$,
et qui divise~$\phi_N$ dans~$F[T]$. Posons $\psi_N=\phi_N/P_N$.
Lorsque $N\to+\infty$, 
les polynômes~$P_N$ convergent dans~$F[T]$ 
vers un polynôme~$P$ dont le polygone de Newton
est affine, de pente~$\mu$,
les polynômes~$\psi_N$ convergent dans~$F\langle T/R\rangle$
vers une série formelle~$\psi$ et l'on a $\phi=P\psi$.
Par ailleurs, les pentes de~$\psi_N$ qui sont~$<\mu$
sont celles de~$\phi$, et les autres sont~$\geq \mu'$.
En passant à la limite, on voit que~$\mu$ n'est pas une pente de~$\psi$.
\qedhere
\end{enumerate}
\end{proof}

\begin{prop}
Soit $\phi\in F\langle T/R\rangle$, 
soit $P\in F[T]$ un polynôme non nul; 
en notant $P=\sum_{k=0}^d c_k T^k$, on suppose
que $\norm P_R=\abs{c_d}R^d$.
Il existe un unique couple $(\rho,\psi)$,
où $\rho\in F[T]$ est un polynôme de degré~$<d$ et
$\psi\in F\langle T/R\rangle$, tel que $\phi = P \psi + \rho$.
De plus, on a  $\norm \rho_R\leq\norm\phi_R$ et 
$\norm P_R \norm \psi_R \leq \norm\phi_R$.
\end{prop}

\begin{proof}
Avant d'établir l'existence et l'unicité, 
démontrons les inégalités qui concluent cet énoncé. 
Considérons une relation $\phi=\psi P + \rho$, où $\rho\in F[T]$
est un polynôme de degré~$<d$ et $\psi\in F\langle T/R\rangle$.
Les inégalités à démontrer sont évidentes lorsque $\psi=0$;
supposons donc $\psi\neq0$,
et considérons le plus grand entier, $n$, tel que $\norm\psi_R$ soit donné
par le coefficient de~$\psi$ d'exposant~$n$; alors
$\norm{\phi-\rho}_R=\norm{\psi P}_R$ est donné par le coefficient d'exposant~$n+d$.
Puisque $\deg(\rho)<d$, on a donc $\norm{\phi-\rho}_R \leq \norm{\phi}_R$,
d'où $\norm\psi_R \norm P_R\leq \norm\phi_R$. 
Puis $\norm \rho_R=\norm{\phi-\psi P}\leq \sup(\norm \phi_R,\norm{\psi P}_R)
\leq \norm\phi_R$. 

Établissons maintenant l'unicité d'une décomposition
comme dans l'énoncé. Si $(\rho,\psi)$ et $(\rho',\psi')$
sont deux tels couples, on a $0 = (\psi'-\psi)P+(\rho'-\rho)$.
Les inégalités précédentes entraînent $\psi'-\psi=0$ puis $\rho'-\rho=0$.

Démontrons maintenant l'existence d'un tel couple $(\rho,\psi)$.
Pour tout entier~$N$, notons $\phi_N$ l'unique polynôme de~$F[T]$
de degré~$\leq N$ tel que $\ord(\phi-\phi_N)>N$.
Considérons alors la division euclidienne de~$\phi_N$ par~$P$,
disons $\phi_N = Q_NP +R_N$.
Pour $N\geq M>d$, on a  donc
\[ \phi_N-\phi_M = (Q_N-Q_M)P + (R_N-R_M). \]
Par conséquent,  on a 
\[ \norm{Q_N-Q_M}_R \norm P_R ,  \norm{R_N-R_M}\leq \norm{\phi_N-\phi_M}. \]
%
 
Lorsque $M,N$ tendent vers l'infini, $\norm{\phi_M-\phi_N}_R$ tend vers~$0$,
donc $\norm{Q_N-Q_M}_R$ tend vers~$0$, de même que $\norm{R_M-R_N}_R$.
Par suite, la suite $(R_N)$ converge vers un polynôme~$\rho$
de degré~$<d$, et la suite~$(Q_N)$ converge vers un élément~$\psi$
de~$F\langle T/R\rangle$. En passant à la limite, on a $\phi=\psi Q+\rho$.
\end{proof}

\begin{prop}[Théorème de « préparation » de Weierstrass]
Soit $\phi\in F\langle T/R\rangle$, notons $\phi=\sum_{n=0}^\infty c_nT^n$,
et soit $N$ le plus grand entier~$n$ tel que $\abs{c_N}R^N=\norm\phi_R$.
Il existe un unique couple $(P,\psi)$ tel que $\phi=\psi P$, 
où $P$ est polynôme unitaire de degré~$N$ dans~$F[T]$ 
et $\psi$ est un élément inversible $\psi\in F\langle T/R\rangle$.
\end{prop}
\begin{proof}
C'est une variante de l'argument fait pour isoler la partie
de pente donnée dans le polygone de Newton. Pour simplifier
les notations, on suppose $c_N=1$.
Pour tout entier~$M$, soit $\phi_M=\sum_{n=0}^M c_nT^n$.
Si $M>N$, le point $(N,-\log(\abs{c_N}))$ 
est un sommet du polygone de Newton du polynôme~$\phi_M$ 
et il existe un unique polynôme unitaire $P_M\in F[T]$ de degré~$N$
divisant~$\phi_M$ et 
dont le polygone de Newton coïncide avec celui de~$\phi$
sur l'intervalle~$[0;N]$. Posons $\psi_M=\phi_M/P_M$.
Lorsque l'entier~$M$ tend vers l'infini, les suites $(P_M)$ et~$(\psi_M)$
convergent vers un polynôme unitaire $P\in F[T]$ 
et vers une série formelle $\psi\in F\langle T/R\rangle$.
Le polygone de Newton de~$P$ coïncide avec celui de~$\phi$ sur~$[0;N]$.
En particulier, si l'on note $P=\sum_{n=0}^N a_nT^n$,
on a $\abs{a_n}R^n\leq \abs{a_N}R^n$ pour tout $n\in\{0,\dots,N\}$.
Notons $\psi=\sum_{n=0}^\infty b_nT^n$. 
Si $n$ est le plus grand entier tel que $\abs{b_n}R^n=\norm\psi_R$,
on a $\abs{c_{N+n}}R^{N+n}=\norm\psi_R\norm P_R=\norm\phi_R$,
de sorte que $n=0$. Compte
tenu de la proposition~\ref{prop.tate-inversible}, 
cela prouve que $\psi$ est inversible dans~$F\langle T/R\rangle$.
\end{proof}

\begin{coro}[\citealp{Strassmann-1928}]
Un élément non nul  de $F\langle T/R\rangle$
n'a qu'un nombre fini de zéros dans~$B(0;R)$.
\end{coro}
\begin{proof}
Soit $\phi$ un élément non nul de~$F\langle T/R\rangle$.
Avec les notations de la proposition, 
les zéros de~$\phi$ sont les zéros de~$P$,
car un élément inversible de~$F\langle T/R\rangle$
ne s'annule pas sur $B(0;R)$.
\end{proof}
\begin{rema}
Soit $\phi$ un élément non nul de $F\langle T/R\rangle$.
On démontre comme en analyse complexe que les zéros
de~$\phi$ dans la boule circonférenciée $B(0;R)$ sont isolés.
Si $F$ est localement compact (et sa valeur absolue non triviale),
cette boule $B(0;R)$ est compacte, si bien 
que l'ensemble des zéros de~$\phi$ dans~$B(0;R)$ est fini.
Cette conclusion n'est cependant pas évidente dans le cas général
et on peut voir ce théorème de Strassmann 
comme une trace de la  compacité du « disque analytique » de rayon~$R$
tel que défini, par exemple, dans la théorie de~\cite{Berkovich-1990}.
\end{rema}

\bibliography{Newton}

\end{document}